\documentclass[11pt]{amsart}
\usepackage{amssymb
,amsthm
,amsmath
,amscd
,mathtools
}
\usepackage[all]{xy}
\usepackage[top=35truemm,bottom=35truemm,left=30truemm,right=30truemm]{geometry}

\newcommand{\Z}{\mathbb{Z}}
\newcommand{\R}{\mathbb{R}}
\newcommand{\C}{\mathbb{C}}

\renewcommand{\AA}{\mathcal{A}}
\newcommand{\RR}{\mathcal{R}}

\newcommand{\id}{\mathrm{id}}
\newcommand{\Rep}{\mathrm{Rep}}
\newcommand{\Irr}{\mathrm{Irr}}
\newcommand{\disc}{\mathrm{disc}}
\newcommand{\gp}{\mathrm{gp}}
\newcommand{\temp}{\mathrm{temp}}
\newcommand{\gen}{\mathrm{gen}}
\newcommand{\Ind}{\mathrm{Ind}}
\newcommand{\Jac}{\mathrm{Jac}}
\newcommand{\Frob}{\mathrm{Frob}}
\newcommand{\semi}{\mathrm{s.s.}}
\newcommand{\St}{\mathrm{St}}
\newcommand{\Ad}{\mathrm{Ad}}
\newcommand{\sub}{\mathrm{sub}}
\newcommand{\diag}{\mathrm{diag}}

\newcommand{\GL}{\mathrm{GL}}
\newcommand{\SL}{\mathrm{SL}}
\newcommand{\SO}{\mathrm{SO}}
\newcommand{\Sp}{\mathrm{Sp}}

\newcommand{\iif}{&\quad&\text{if }}
\newcommand{\other}{&\quad&\text{otherwise}}
\newcommand{\resp}{resp.~}
\renewcommand{\1}{\mathbf{1}}

\newcommand{\pair}[1]{\left\langle #1 \right\rangle}
\newcommand{\half}[1]{\frac{#1}{2}}
\newcommand{\ub}[1]{\underline{#1}}

\newtheorem{thm}{Theorem}[section]
\newtheorem{lem}[thm]{Lemma}
\newtheorem{prop}[thm]{Proposition}
\newtheorem{cor}[thm]{Corollary}
\newtheorem{rem}[thm]{Remark}
\newtheorem{ex}[thm]{Example}

\title{Jacquet modules and local Langlands correspondence}
\author{Hiraku Atobe}
\date{}
\address{
Department of Mathematics, Hokkaido University
Kita 10, Nishi 8, Kita-Ku, Sapporo, Hokkaido, 060-0810, Japan 
}
\email{
atobe@math.sci.hokudai.ac.jp
}

\allowdisplaybreaks

\begin{document}
\maketitle

\begin{abstract}
In this paper, 
we explicitly compute the semisimplifications of all Jacquet modules of 
irreducible representations with generic $L$-parameters
of $p$-adic split odd special orthogonal groups or symplectic groups.
Our computation represents them in terms of linear combinations of standard modules with rational coefficients.
The main ingredient of this computation 
is to apply M{\oe}glin's explicit construction of local $A$-packets to tempered $L$-packets.
\end{abstract}

\section{Introduction}
When $G$ is a $p$-adic reductive group and $P=MN$ is a parabolic subgroup, 
there is the normalized induction functor
\[
\Ind_P^G \colon \Rep(M) \rightarrow \Rep(G).
\]
The (normalized) Jacquet functor 
\[
\Jac_P \colon \Rep(G) \rightarrow \Rep(M)
\]
is the left adjoint functor of $\Ind_P^G$.
For $\pi \in \Rep(G)$, 
the object $\Jac_P(\pi) \in \Rep(M)$ is called the \textbf{Jacquet module} of $\pi$ with respect to $P$.
In the representation theory of $p$-adic reductive groups, 
the induction functors and the Jacquet functors
are ones of the most basic and important terminologies.
One of the reasons why they are so important 
is that they are both exact functors.
\par

The Jacquet modules have many applications.
For example: 
\begin{itemize}
\item
Looking at the Jacquet modules of irreducible representation $\pi$ of $G$, 
one can take a parabolic subgroup $P=MN$ and an irreducible supercuspidal representation $\rho_M$ of $M$
such that $\pi \hookrightarrow \Ind_P^G(\rho_M)$.
Such a $\rho_M$ is called the \textbf{cuspidal support} of $\pi$.

\item
\textbf{Casselman's criterion} says that the growth of matrix coefficients of an irreducible representation $\pi$
is determined by exponents of the Jacquet modules of $\pi$.

\item
M{\oe}glin explicitly constructed the \textbf{local $A$-packets}, 
which are the ``local factors of Arthur's global classification'', 
by taking Jacquet functors intelligently.
\end{itemize}
\vskip10pt

In this paper, we shall give an explicit description of 
the semisimplifications of Jacquet modules of tempered representations
of split odd special orthogonal groups $\SO_{2n+1}(F)$ or symplectic groups $\Sp_{2n}(F)$, 
where $F$ is a non-archimedean local field of characteristic zero.
To do this, it is necessary to have some sort of classification of irreducible representations of these groups. 
We use the local Langlands correspondence established by Arthur \cite{Ar} for such a classification.
\par

The local Langlands correspondence attaches each irreducible representation $\pi$ 
of $G(F) = \SO_{2n+1}(F)$ or $G(F) = \Sp_{2n}(F)$
to its $L$-parameter $(\phi, \eta)$, 
where
\[
\phi \colon W_F \times \SL_2(\C) \rightarrow \GL_N(\C)
\]
is a self-dual representation of the Weil--Deligne group $W_F \times \SL_2(\C)$ 
with a suitable structure, 
and 
\[
\eta \in \Irr(A_\phi)
\]
is an irreducible character of the component group $A_\phi$ associated to $\phi$ 
(which is trivial on the central element).
\vskip10pt

The Jacquet modules will be computed by 
two main theorems (Theorems \ref{jac1} and \ref{jac2}) 
and Tadi{\'c}'s formula (Theorem \ref{tdc}) 
together with Lemma \ref{circ}.
Fix an irreducible unitary supercuspidal representation $\rho$ of $\GL_d(F)$.
By abuse of notation, we denote by the same notation $\rho$ 
the irreducible representation of $W_F$ corresponding to $\rho$
by the local Langlands correspondence.
Let $P_d = M_dN_d$ be the standard parabolic subgroup of $G(F)$ 
with Levi subgroup $M_d \cong \GL_d(F) \times G_0(F)$ 
for some classical group $G_0$ of the same type as $G$.
For an irreducible representation $\pi$, 
if the semisimplification of $\Jac_{P_d}(\pi)$ is of the form
\[
\semi\Jac_{P_d}(\pi) = \bigoplus_{i \in I} \tau_i \boxtimes \pi_i, 
\]
we set 
\[
\Jac_{\rho|\cdot|^x}(\pi) = \bigoplus_{\substack{i \in I \\ \tau_i \cong \rho|\cdot|^x}} \pi_i
\]
for $x \in \R$.
The first main theorem is the description of $\Jac_{\rho|\cdot|^x}(\pi)$ for tempered $\pi$ (Theorem \ref{jac1}).
To state this theorem clearly, 
we introduce an \textbf{enhanced component group} $\AA_\phi$ attached to $\phi$
with a canonical surjection $\AA_\phi \twoheadrightarrow A_\phi$
in \S \ref{Lpara}.
\par

For discrete series $\pi$, 
Theorem \ref{jac1} has been proven by Xu \cite[Lemma 7.3]{X1}
to describe the cuspidal support of $\pi$ in terms of its $L$-parameter.
As related works, Aubert--Moussaoui--Solleveld \cite{AMS1, AMS2, AMS3} 
defined the ``cuspidality'' of $L$-parameters $(\phi, \eta)$ by a geometric way, 
and compared this notion with the cuspidal supports or the Bernstein components of corresponding $\pi$.
Theorem \ref{jac1} gives us more information for $\pi$ than its cuspidal support. 
The main ingredient for the proof of Theorem \ref{jac1} 
is M{\oe}glin's explicit construction of tempered $L$-packets (Theorem \ref{moe}).
\vskip10pt

The second main theorem (Theorem \ref{jac2}) is a reduction of 
the computation of the Jacquet module $\semi\Jac_{P_k}(\pi)$
with respect to any maximal parabolic subgroup $P_k$
to the one of $\Jac_{\rho|\cdot|^x}(\pi)$.
Using Theorems \ref{jac1} and \ref{jac2} (together with Lemma \ref{circ}), 
we can explicitly compute the semisimplifications of all Jacquet modules of 
irreducible tempered representations $\pi$.
In fact, using a generalization of the standard module conjecture 
by M{\oe}glin--Waldspurger (Theorem \ref{smc})
and Tadi{\'c}'s formula (Theorem \ref{tdc}),  
we can apply this explicit computation to 
any irreducible representation $\pi$ with generic $L$-parameter $(\phi, \eta)$.
\vskip10pt

This paper is organized as follows.
In \S \ref{ind-jac}, 
we review some basic results on induced representations and Jacquet modules for classical groups.
In particular, Tadi{\'c}'s formula, which computes the Jacquet modules of induced representations, 
is stated in \S \ref{so-sp}.
In \S \ref{s.LLC}, we explain the local Langlands correspondence 
and M{\oe}glin's explicit construction of tempered $L$-packets.
In \S \ref{desc}, we state the main theorems (Theorems \ref{jac1} and \ref{jac2}) and give some examples.
Finally, we prove the main theorems in \S \ref{pf}

\subsection*{Notation}
Let $F$ be a non-archimedean local field of characteristic zero.
We denote by $W_F$ the Weil group of $F$.
The norm map $|\cdot| \colon W_F \rightarrow \R^\times$ is normalized so that $|\Frob| = q^{-1}$, 
where $\Frob \in W_F$ is a fixed (geometric) Frobenius element, 
and $q = q_F$ is the cardinality of the residual field of $F$.
\par

Each irreducible supercuspidal representation $\rho$ of $\GL_d(F)$
is identified with the irreducible bounded representation of $W_F$ of dimension $d$ 
via the local Langlands correspondence for $\GL_d$.
Through this paper, we fix such a $\rho$.
For each positive integer $a$, the unique irreducible algebraic representation of $\SL_2(\C)$
of dimension $a$ is denoted by $S_a$.
\par

For a $p$-adic group $G$, we denote by $\Rep(G)$ (\resp $\Irr(G)$) 
the set of equivalence classes of smooth admissible (\resp irreducible) representations of $G$.
For $\Pi \in \Rep(G)$, we write $\semi(\Pi)$ for the semisimplification of $\Pi$.

\section{Induced representations and Jacquet modules}\label{ind-jac}
In this section, we recall some results on induced representations and Jacquet modules.

\subsection{Representations of $\GL_k(F)$}
Let $P=MN$ be a standard parabolic subgroup of $\GL_k(F)$, 
i.e., $P$ contains the Borel subgroup consisting of upper half triangular matrices.
Then the Levi subgroup $M$ is isomorphic to $\GL_{k_1}(F) \times \dots \times \GL_{k_r}(F)$
with $k_1 + \dots + k_r = k$.
For smooth representations $\tau_1, \dots, \tau_r$ of $\GL_{k_1}(F), \dots, \GL_{k_r}(F)$, respectively, 
we denote the normalized induced representation by
\[
\tau_1 \times \dots \times \tau_r \coloneqq 
\Ind_{P}^{\GL_k(F)}(\tau_1 \boxtimes \dots \boxtimes \tau_r).
\]
\par
A \textbf{segment} is a symbol $[x,y]$, 
where $x,y \in \R$ with $x-y \in \Z$ and $x \geq y$.
We identify $[x,y]$ with the set $\{x, x-1, \dots, y\}$ 
so that $\#[x,y] = x-y+1$.
Then the normalized induced representation 
\[
\rho|\cdot|^{x} \times \dots \times \rho|\cdot|^y
\]
of $\GL_{d(x-y+1)}(F)$
has a unique irreducible subrepresentation, 
which is denoted by 
\[
\pair{\rho; x, \dots, y}.
\]
If $y = -x \leq 0$, this is called a \textbf{Steinberg representation} 
and is denoted by 
\[
\St(\rho, 2x+1) = \pair{\rho; x, \dots, -x}, 
\]
which is a discrete series representation of $\GL_{d(2x+1)}(F)$. 
In general, $\pair{\rho; x, \dots, y}$ is the twist $|\cdot|^{\half{x+y}} \St(\rho, x-y+1)$.
We say that 
two segments $[x, y]$ and $[x', y']$ are \textbf{linked} if 
$[x,y] \not\subset [x',y']$, $[x',y'] \not\subset [x,y]$ as sets, 
and $[x,y] \cup [x,y']$ is also a segment.
The linked-ness gives an irreducibility criterion for induced representations.
\begin{thm}[Zelevinsky {\cite[Theorem 9.7]{Z}}]\label{zel}
Let $[x,y]$ and $[x',y']$ be segments, 
and let $\rho$ and $\rho'$ be irreducible unitary supercuspidal representations 
of $\GL_d(F)$ and $\GL_{d'}(F)$, respectively.
Then the induced representation 
\[
\pair{\rho; x, \dots, y} \times \pair{\rho'; x', \dots, y'}
\]
is irreducible unless $[x, y]$ are $[x',y']$ are linked, and $\rho \cong \rho'$.
\end{thm}

Let $\Irr_\rho(\GL_{dm}(F))$ be the subset of $\Irr(\GL_{dm}(F))$ consisting of $\tau$ 
with cuspidal support of the form $\rho|\cdot|^{x_1} \times \dots \times \rho|\cdot|^{x_m}$, i.e., 
\[
\tau \hookrightarrow \rho|\cdot|^{x_1} \times \dots \times \rho|\cdot|^{x_m}
\] 
for some $x_1, \dots, x_m \in \R$.
We understand that $\1 \coloneqq \1_{\GL_0} \in \Irr_\rho(\GL_0(F))$.
It is easy to see that 
\begin{itemize}
\item
for pairwise distinct irreducible unitary supercuspidal representations $\rho_1, \dots, \rho_r$, 
if $\tau_i \in \Irr_{\rho_i}(\GL_{d_im_i}(F))$ for $i = 1, \dots, r$, 
then the induced representation $\tau_1 \times \dots \times \tau_r$ is irreducible; 
\item
any irreducible representation of $\GL_k(F)$ is of the above form for some $\tau_i \in \Irr_{\rho_i}(\GL_{d_im_i}(F))$.
\end{itemize}

\begin{lem}\label{GL}
Let $\Omega_m$ be the subset of $\R^m$ consisting of elements
\[
\ub{x} = 
( \underbrace{x_1, x_1-1, \dots, y_1}_{x_1-y_1+1}, \underbrace{x_2, x_2-1, \dots, y_2}_{x_2-y_2+1}, 
\dots, \underbrace{x_t, x_t-1, \dots, y_t}_{x_t-y_t+1} )
\]
such that $x_{i-1} \leq x_{i}$ for $1 < i \leq t$, and $y_{i-1} \leq y_{i}$ if $x_{i-1} = x_i$.
Let $\Delta_i = \pair{\rho; x_i, x_i-1, \dots, y_i}$ be the discrete series representation of $\GL_{d(x_i-y_i+1)}(F)$
corresponding to the segment $[x_i,y_i]$.
Then the induced representation 
$\Delta_{\ub{x}} \coloneqq \Delta_1 \times \dots \times \Delta_t$ 
has a unique irreducible subrepresentation $\tau_{\ub{x}}$.
The map $\ub{x} \mapsto \tau_{\ub{x}}$ gives a bijection
\[
\Omega_m \rightarrow \Irr_\rho(\GL_{dm}(F)).
\]
\end{lem}
\begin{proof}
This follows from the Langlands classification and Theorem \ref{zel}, 
See also \cite[Proposition 9.6]{Z}.
\end{proof}

For a partition $(k_1,\dots, k_r)$ of $k$, 
we denote by $\Jac_{(k_1, \dots, k_r)}$ the normalized Jacquet functor on $\Rep(\GL_k(F))$
with respect to the standard maximal parabolic subgroup $P = MN$ 
with $M \cong \GL_{k_1}(F) \times \dots \times \GL_{k_r}(F)$.
The Jacquet module of $\pair{\rho; x, \dots, y}$ with respect to a maximal parabolic subgroup 
is computed by Zelevinsky.

\begin{prop}[{\cite[Proposition 9.5]{Z}}]
Suppose that $x \not= y$ and set $k = d(x-y+1)$.
Then $\Jac_{(k_1, k_2)}(\pair{\rho; x, \dots, y}) = 0$ unless $k_1 \equiv 0 \bmod d$.
If $k_1 = dm$ with $1 \leq m \leq x-y$, we have
\[
\Jac_{(k_1, k_2)}(\pair{\rho; x, \dots, y}) 
= \pair{\rho; x, \dots, x-(m-1)} \boxtimes \pair{\rho; x - m, \dots, y}. 
\]
\end{prop}

If 
\[
\semi\Jac_{(d, k-d)}(\tau) = \bigoplus_{i \in I} \tau_i \boxtimes \tau'_i, 
\]
for $x \in \R$, we set 
\[
\Jac_{\rho|\cdot|^x}(\tau) = \bigoplus_{\substack{i \in I\\ \tau_i \cong \rho|\cdot|^x }}\tau'_i.
\]
For $\ub{x} = (x_1, \dots, x_r) \in \R^r$, we also define 
\[
\Jac_{\rho|\cdot|^{\ub{x}}} = \Jac_{\rho|\cdot|^{x_r}} \circ \dots \circ \Jac_{\rho|\cdot|^{x_1}}.
\]
This is a functor 
\[
\Jac_{\rho|\cdot|^{\ub{x}}} \colon \Rep(\GL_k(F)) \rightarrow \Rep(\GL_{k-dr}(F)).
\]
In particular, 
when $\tau \in \Rep(\GL_{dm}(F))$ is of finite length, for $\ub{x} = (x_1, \dots, x_m) \in \R^m$, 
the Jacquet module $\Jac_{\rho|\cdot|^{\ub{x}}}(\tau)$ is 
a representation of the trivial group $\GL_0(F)$ of finite length so that
it is a finite dimensional $\C$-vector space.

\begin{lem}\label{dim}
Let $\ub{x} = (x_1, \dots, y_1, \dots, x_t, \dots, y_t) \in \Omega_m$ 
such that $x_{i-1} \leq x_{i}$ for $1 < i \leq t$, and $y_{i-1} \leq y_{i}$ if $x_{i-1} = x_i$
as in Lemma \ref{GL}.
For $(x,y) \in \{(x_i,y_i)\}_i$, 
if we set $m_{(x,y)} = \#\{i \ |\ (x_i,y_i) = (x,y)\}$, 
then for $\ub{y} \in \Omega_m$, we have
\[
\dim_{\C} \Jac_{\rho|\cdot|^{\ub{y}}}(\tau_{\ub{x}}) = 
\left\{
\begin{aligned}
&\prod_{(x,y) \in \{(x_i,y_i)\}_i} m_{(x,y)}! \iif \ub{x'} = \ub{x}, \\
&0 \iif \ub{y} < \ub{x}.
\end{aligned}
\right.
\]
Here, we regard $\R^m$ as a totally ordered set with respect to the lexicographical order.
\end{lem}
\begin{proof}
Fix $x \in \R$.
We note that 
$\Jac_{\rho|\cdot|^x}(\Delta_{\ub{x}}) \not= 0$ if and only if $x \in \{x_1, \dots, x_t\}$.
Let $\ub{x'_1}, \dots, \ub{x'_{l}} \in \Omega_{m-1}$ be the elements 
obtained by removing $x$ from a component of $\ub{x}$, and rearranging it
(so that $l = \#\{i\ |\ x_i = x\}$).
Then $\Jac_{\rho|\cdot|^x}(\Delta_{\ub{x}}) = \sum_{i=1}^l \Delta_{\ub{x'_i}}$.
Using this, we obtain the lemma by induction on $m$.
\if()
\par

We write
\[
\Delta_1 \times \dots \times \Delta_k \cong 
\underbrace{(\Delta_1' \times \dots \times \Delta_1')}_{m_1} 
\times \dots \times 
\underbrace{(\Delta_l' \times \dots \times \Delta_l')}_{m_l}  
\]
with $\Delta_i' = \pair{\rho; x_i', \dots, y_i'}$ such that 
$x'_{i-1} \leq x'_{i}$ for $1 < i \leq l$,
and $y'_{i-1} < y'_i$ if $x'_{i-1} = x'_{i}$.
Since $\tau_{\ub{x}}$ is contained in this induced representation, 
we have
\[
\Jac_{\rho|\cdot|^{\ub{x}}}(\tau_{\ub{x}}) \subset 
\Jac_{\rho|\cdot|^{\ub{x}}}(
(\Delta_1' \times \dots \times \Delta_1')
\times \dots \times 
(\Delta_l' \times \dots \times \Delta_l')).
\]
Moreover, we have
\[
\semi\Jac_{(d(x'_1-y'_1+1)m_1, \dots, d(x'_l-y'_l+1)m_l)}(\tau_{\ub{x}}) \supset 
(\Delta_1' \times \dots \times \Delta_1')
\otimes \dots \otimes 
(\Delta_l' \times \dots \times \Delta_l')
\]
since the right hand side is irreducible.
This implies that 
\[
\Jac_{\rho|\cdot|^{\ub{x}}}(\tau_{\ub{x}}) \supset 
\Jac_{\rho|\cdot|^{\ub{x}}}(
(\Delta_1' \times \dots \times \Delta_1')
\otimes \dots \otimes 
(\Delta_l' \times \dots \times \Delta_l')).
\]
Since 
\begin{align*}
&\dim_\C \Jac_{\rho|\cdot|^{\ub{x}}}(
(\Delta_1' \times \dots \times \Delta_1')
\times \dots \times 
(\Delta_l' \times \dots \times \Delta_l'))
\\&=
\dim_\C \Jac_{\rho|\cdot|^{\ub{x}}}(
(\Delta_1' \times \dots \times \Delta_1')
\otimes \dots \otimes 
(\Delta_l' \times \dots \times \Delta_l'))
=
\prod_{j=1}^l m_j!, 
\end{align*}
we have $\dim_\C \Jac_{\rho|\cdot|^{\ub{x}}}(\tau_{\ub{x}}) = \prod_{j=1}^l m_j!$.
\par

As in the proof of \cite[Proposition 9.6]{Z}, 
the element $\ub{x} \in \Omega_m$ is characterized by 
the lowest element in the set
\begin{align*}
&\{\ub{x'} \in \R^m\ |\ \Jac_{\rho|\cdot|^{\ub{x'}}}(\tau_{\ub{x}}) \not= 0\}
\end{align*}
with respect to the lexicographical order.
Hence we obtain the last assertion.
\fi
\end{proof}

When $\ub{y} > \ub{x}$, 
one can also compute $\dim_{\C} \Jac_{\rho|\cdot|^{\ub{y}}}(\tau_{\ub{x}})$ inductively.
\par

Let $\RR_k$ be the Grothendieck group of the category of 
smooth representations of $\GL_k(F)$ of finite length.
By the semisimplification, 
we identify the objects in this category with elements in $\RR_k$.
Elements in $\Irr(\GL_k(F))$ form a $\Z$-basis of $\RR_k$.
Set $\RR = \oplus_{k \geq 0} \RR_k$.
The induction functor gives a product 
\[
m \colon \RR \otimes \RR \rightarrow \RR, \ 
\tau_1 \otimes \tau_2 \mapsto \semi(\tau_1 \times \tau_2). 
\]
This product makes $\RR$ an associative commutative ring.
On the other hand, the Jacquet functor gives a coproduct
\begin{align*}
m^* \colon &\RR \rightarrow \RR \otimes \RR
\end{align*}
which is defined by the $\Z$-linear extension of 
\[
\Irr(\GL_k(F)) \ni \tau \mapsto \sum_{k_1 = 0}^k \semi \Jac_{(k_1, k-k_1)}(\tau).
\]
Then $m$ and $m^*$ make $\RR$ a graded Hopf algebra, 
i.e., $m^* \colon \RR \rightarrow \RR \otimes \RR$ is a ring homomorphism.

\subsection{Representations of $\SO_{2n+1}$ and $\Sp_{2n}$}\label{so-sp}
We set $G$ to be split $\SO_{2n+1}$ or $\Sp_{2n}$, 
i.e., $G$ is the split algebraic group of type $B_n$ or $C_n$.
Fix a Borel subgroup of $G(F)$, 
and let $P=MN$ be a standard parabolic subgroup of $G(F)$.
Then the Levi part $M$ is of the form 
$\GL_{k_1}(F) \times \dots \times \GL_{k_r}(F) \times G_{0}(F)$
with $G_0 = \SO_{2n_0+1}$ or $G_0 = \Sp_{2n_0}$ such that $k_1+\dots+k_r+n_0=n$.
For a smooth representation $\tau_1 \boxtimes \dots \boxtimes \tau_r \boxtimes \pi_0$ of $M$, 
we denote the normalized induced representation by 
\[
\tau_1 \times \dots \times \tau_r \rtimes \pi_0 
= \Ind_P^{G(F)}( \tau_1 \boxtimes \dots \boxtimes \tau_r \boxtimes \pi_0 ).
\]
The functor $\Ind_{P}^{G(F)} \colon \Rep(M) \rightarrow \Rep(G(F))$ is exact. 
\par

On the other hand, for a smooth representation $\pi$ of $G(F)$, 
we denote the normalized Jacquet module with respect to $P$ by 
\[
\Jac_P(\pi), 
\]
and its semisimplification by $\semi\Jac_P(\pi)$.
The functor $\Jac_{P} \colon \Rep(G(F)) \rightarrow \Rep(M)$ is exact. 
The Frobenius reciprocity asserts that 
\[
\mathrm{Hom}_{G(F)}(\pi, \Ind_P^{G(F)}(\sigma)) \cong \mathrm{Hom}_{M}(\Jac_P(\pi), \sigma)
\]
for $\pi \in \Rep(G(F))$ and $\sigma \in \Rep(M)$.

The maximal parabolic subgroup with Levi $\GL_{k}(F) \times G_{0}(F)$ 
is denoted by $P_k = M_k N_k$.
If
\[
\semi\Jac_{P_d}(\pi) = \bigoplus_{i \in I} \tau_i \boxtimes \pi_i, 
\]
for a real number $x$, 
we set
\[
\Jac_{\rho|\cdot|^x}(\pi) = \bigoplus_{\substack{i \in I\\ \tau_i \cong \rho|\cdot|^x }}\pi_i.
\]
This is a representation of $G_{0}(F)$, which is $\SO_{2(n-d)+1}(F)$ or $\Sp_{2(n-d)}(F)$.
Also, for $\ub{x} = (x_1, \dots, x_r) \in \R^r$, we set
\[
\Jac_{\rho|\cdot|^{\ub{x}}}(\pi) = 
\Jac_{\rho|\cdot|^{x_1}, \dots, \rho|\cdot|^{x_r}}(\pi) = 
\Jac_{\rho|\cdot|^{x_r}} \circ \dots \circ \Jac_{\rho|\cdot|^{x_1}}(\pi).
\]
We recall some properties of $\Jac_{\rho|\cdot|^{x_1}, \dots, \rho|\cdot|^{x_r}}$.

\begin{lem}[{\cite[Lemmas 5.3, 5.6]{X1}}]\label{lem5}
Let $\pi$ be an irreducible representation of $G(F)$.

\begin{enumerate}
\item
Suppose that $\Jac_{\rho|\cdot|^{x_1}, \dots, \rho|\cdot|^{x_r}}(\pi)$ is nonzero.
Then there exists an irreducible constituent $\sigma$ of $\Jac_{\rho|\cdot|^{x_1}, \dots, \rho|\cdot|^{x_r}}(\pi)$ 
such that we have an inclusion
\[
\pi \hookrightarrow \rho|\cdot|^{x_1} \times \dots \times \rho|\cdot|^{x_r} \rtimes \sigma.
\]

\item
If $|x-y| \not= 1$, then $\Jac_{\rho|\cdot|^{x}, \rho|\cdot|^{y}}(\pi) = \Jac_{\rho|\cdot|^{y}, \rho|\cdot|^{x}}(\pi)$.
\end{enumerate}
\end{lem}
\par

Let $\RR_n(G)$ be the Grothendieck group of the category of 
smooth representations of $G(F)$ of finite length, 
where $n = \mathrm{rank}(G)$, i.e., $G = \SO_{2n+1}$ or $G = \Sp_{2n}$.
Set $\RR(G) = \oplus_{n \geq 0} \RR_n(G)$.
The parabolic induction defines a module structure
\[
\rtimes \colon \RR \otimes \RR(G) \rightarrow \RR(G), \ \tau \otimes \pi \mapsto \semi(\tau \rtimes \pi), 
\]
and the Jacquet functor defines a comodule structure
\[
\mu^* \colon \RR(G) \rightarrow \RR \otimes \RR(G)
\]
by
\[
\Irr(G(F)) \ni \pi \mapsto \sum_{k=0}^{\mathrm{rank}(G)} \semi\Jac_{P_k}(\pi).
\]
\par

When 
\[
\mu^*(\pi) = \sum_{i \in I} \tau_i \otimes \pi_i, 
\]
we define $\mu^*_\rho(\pi)$ by 
\[
\mu^*_\rho(\pi) = \sum_{\substack{i \in I \\ \tau_i \in \Irr_{\rho}(\GL_{dm}(F))}} \tau_i \otimes \pi_i.
\]

\begin{lem}\label{circ}
If we define $\iota \colon \RR(G) \rightarrow \RR \otimes \RR(G)$ by $\pi \mapsto \1_{\GL_0(F)} \otimes \pi$, 
we have
\[
\mu^* = \circ_{\rho} \left( (m \otimes \id) \circ (\id \otimes \mu^*_\rho) \right) \circ \iota, 
\]
where $\rho$ runs over all irreducible unitary supercuspidal representations of $\GL_d(F)$ 
for $d > 0$.
\end{lem}
\begin{proof}
Fix an irreducible representation $\pi$ of $G(F)$.
First, we note that there are only finitely many $\rho$ such that $\mu^*_\rho(\pi) \not= 0$.
Second, we claim that
\[
(m \otimes \id) \circ (\id \otimes \mu^*_{\rho'}) \circ \mu^*_{\rho}(\pi)
= 
(m \otimes \id) \circ (\id \otimes \mu^*_{\rho}) \circ \mu^*_{\rho'}(\pi)
\]
for distinct $\rho$ and $\rho'$.
In fact, this is the sum of subrepresentations appearing $\mu^*(\pi)$ of the form
\[
(\tau \times \tau') \otimes \pi_0 = (\tau' \times \tau) \otimes \pi_0, 
\]
where $\tau$ is of type $\rho$ and $\tau'$ is of type $\rho'$.
By the same argument, we have 
\[
\mu^* = \circ_{\rho} \left( (m \otimes \id) \circ (\id \otimes \mu^*_\rho) \right) \circ \iota(\pi),
\]
as desired.
\end{proof}

Tadi{\'c} established a formula to compute $\mu^*$ for induced representations.
The contragredient functor $\tau \mapsto \tau^\vee$ defines an automorphism 
$\vee \colon \RR \rightarrow \RR$ in a natural way.
Let $s \colon \RR \otimes \RR \rightarrow \RR \otimes \RR$ be the homomorphism defined by
$\sum_i\tau_i \otimes \tau'_i \mapsto \sum_i \tau'_i \otimes \tau_i$.
\par

\begin{thm}[Tadi{\'c} \cite{T}]\label{tdc}
Consider the composition
\[
M^* = (m \otimes \id) \circ (\vee \otimes m^*) \circ s \circ m^* \colon \RR \rightarrow \RR \otimes \RR.
\]
Then for the maximal parabolic subgroup $P_k=M_kN_k$ of $G(F)$
and for an admissible representation $\tau \boxtimes \pi$ of $M_k$, 
we have
\[
\mu^*(\tau \rtimes \pi) = M^*(\pi) \rtimes \mu^*(\pi).
\]
\end{thm}

In particular, we have the following.
\begin{cor}\label{cor.tdc}
For a segment $[x,y]$, we have
\begin{align*}
&\mu^*_\rho( \pair{\rho; x, \dots, y} \rtimes \pi)
= \sum_{ \substack{k,l \geq 0 \\ k+l \leq x-y+1}} 
\\&
\left((\pair{\rho; \underbrace{-y, \dots, -y-k+1}_{k}} \times \pair{\rho; \underbrace{x,\dots,x-l+1}_{l}}) 
\otimes \pair{\rho; x-l, \dots, y+k} \right)
\rtimes \mu^*_\rho(\pi).
\end{align*}
\end{cor}

\section{Local Langlands correspondence}\label{s.LLC}
In this section, we review the local Langlands correspondence
for split $\SO_{2n+1}$ or $\Sp_{2n}$ over $F$.
\par

\subsection{$L$-parameters}\label{Lpara}
A homomorphism 
\[
\phi \colon W_F \times \SL_2(\C) \rightarrow \GL_k(\C)
\]
is called a \textbf{representation} of the Weil--Deligne group $W_F \times \SL_2(\C)$
if 
\begin{itemize}
\item
$\phi(\Frob) \in \GL_k(\C)$ is semisimple; 
\item
$\phi|W_F$ is smooth, i.e., has an open kernel; 
\item
$\phi|\SL_2(\C)$ is algebraic.
\end{itemize}
We say that a representation $\phi$ of $W_F \times \SL_2(\C)$ is \textbf{tempered}
if $\phi(W_F)$ is bounded.
The local Langlands correspondence for $\GL_k$ asserts that 
there is a canonical bijection 
\[
\Irr(\GL_k(F)) \longleftrightarrow \{\phi \colon W_F \times \SL_2(\C) \rightarrow \GL_k(\C)\}/\cong,
\]
which preserves the tempered-ness.
\par

An \textbf{$L$-parameter} for $\SO_{2n+1}$ is a symplectic representation
\[
\phi \colon W_F \times \SL_2(\C) \rightarrow \Sp_{2n}(\C).
\]
Similarly, an \textbf{$L$-parameter} for $\Sp_{2n}$ is an orthogonal representation
\[
\phi \colon W_F \times \SL_2(\C) \rightarrow \SO_{2n+1}(\C).
\]
For $G = \SO_{2n+1}$ or $G = \Sp_{2n}$, 
we let $\Phi(G)$ be the set of equivalence classes of $L$-parameters for $G$.
We say that 
\begin{itemize}
\item
$\phi \in \Phi(G)$ is \textbf{discrete} 
if $\phi$ is a multiplicity-free sum of irreducible self-dual representations of the same type as $\phi$; 
\item
$\phi \in \Phi(G)$ is \textbf{of good parity} 
if $\phi$ is a sum of irreducible self-dual representations of the same type as $\phi$; 
\item
$\phi \in \Phi(G)$ is \textbf{tempered} if $\phi(W_F)$ is bounded; 
\item
$\phi \in \Phi(G)$ is \textbf{generic} 
if the adjoint $L$-function $L(s,\phi, \Ad)$ is regular at $s=1$.
\end{itemize}
We denote by $\Phi_\disc(G)$ (\resp $\Phi_\gp(G)$, $\Phi_\temp(G)$, and $\Phi_\gen(G)$) 
the subset of $\Phi(G)$ consisting of discrete $L$-parameters 
(\resp $L$-parameters of good parity, tempered $L$-parameters, and generic $L$-parameters).
Then we have inclusions 
\[
\Phi_\disc(G) \subset \Phi_\gp(G) \subset \Phi_\temp(G) \subset \Phi_\gen(G) \subset \Phi(G).
\]
\par

For $\phi \in \Phi(G)$, we can decompose
\[
\phi = m_1\phi_1 \oplus \dots \oplus m_r\phi_r \oplus (\phi' \oplus \phi'^\vee), 
\]
where $\phi_1, \dots, \phi_r$ are distinct irreducible self-dual representations of the same type as $\phi$, 
$m_i \geq 1$ is the multiplicity of $\phi_i$ in $\phi$, 
and $\phi'$ is a sum of irreducible representations 
which are not self-dual or self-dual of the opposite type to $\phi$.
We define the \textbf{component group} $A_\phi$ of $\phi$ by 
\[
A_\phi = \bigoplus_{i=1}^r (\Z/2\Z) \alpha_{\phi_i} \cong (\Z/2\Z)^r.
\] 
Namely, $A_\phi$ is a free $\Z/2\Z$-module of rank $r$ 
and $\{\alpha_{\phi_1}, \dots, \alpha_{\phi_r}\}$ is a basis of $A_\phi$
with $\alpha_{\phi_i}$ associated to $\phi_i$.
We set
\[
z_\phi = \sum_{i=1}^r m_i \alpha_{\phi_i} \in A_\phi
\]
and we call $z_\phi$ the \textbf{central element} in $A_\phi$.
\par

We shall introduce an \textbf{enhanced component group} $\AA_\phi$ associated to $\phi \in \Phi(G)$.
Write $\phi = \phi_\gp \oplus (\phi' \oplus \phi'^\vee)$, 
where $\phi_\gp$ is the sum of irreducible self-dual representations of the same type as $\phi$, 
and $\phi'$ is a sum of irreducible representations which are not of the same type as $\phi$.
We decompose 
\[
\phi_\gp = \bigoplus_{i \in I} \phi_i
\]
into the sum of irreducible representations.
Then we define the \textbf{enhanced component group} $\AA_\phi$ associated to $\phi$ by 
\[
\AA_\phi = \bigoplus_{i \in I} (\Z/2\Z) \alpha_i.
\]
Namely, $\AA_\phi$ is a free $\Z/2\Z$-module 
whose rank is equal to the length of $\phi_\gp$.
By abuse of notation, we put $z_\phi = \sum_{i \in I} \alpha_i \in \AA_\phi$ and call it the \textbf{central element} in $\AA_\phi$.
There is a canonical surjection
\[
\AA_\phi \twoheadrightarrow A_\phi,\ \alpha_i \mapsto \alpha_{\phi_i}.
\]
The kernel of this map is generated by $\alpha_i + \alpha_j$ for $\phi_i \cong \phi_j$.
Moreover, this map preserves the central elements.

\subsection{Local Langlands correspondence}
We denote by $\Irr_\disc(G(F))$ (\resp $\Irr_\temp(G(F))$)
the set of equivalence classes of irreducible discrete series (\resp tempered) representations of $G(F)$.
The local Langlands correspondence established by Arthur is as follows:

\begin{thm}[{\cite[Theorem 2.2.1]{Ar}}]\label{LLC}
Let $G$ be a split $\SO_{2n+1}$ or $\Sp_{2n}$.
\begin{enumerate}
\item
There exists a canonical surjection 
\[
\Irr(G(F)) \rightarrow \Phi(G).
\]
For $\phi \in \Phi(G)$, we denote by $\Pi_\phi$
the inverse image of $\phi$ under this map,
and call $\Pi_\phi$ the \textbf{$L$-packet associated to $\phi$}.
\item
There exists an injection 
\[
\Pi_\phi \rightarrow \widehat{A_\phi},
\]
which satisfies certain endoscopic character identities. 
Here, $\widehat{A_\phi}$ is the Pontryagin dual of $A_\phi$.
The image of this map is 
\[
\{\eta \in \widehat{A_\phi}\ |\ \eta(z_\phi) = 1\}.
\]
When $\pi \in \Pi_\phi$ corresponds to $\eta \in \widehat{A_\phi}$, 
we write $\pi = \pi(\phi, \eta)$.

\item
For $* \in \{\disc, \temp\}$, 
\[
\Irr_*(G(F)) = \bigsqcup_{\phi \in \Phi_*(G)}\Pi_\phi.
\]

\item
Assume that $\phi=\phi_\tau \oplus \phi_0 \oplus \phi_\tau^\vee \in \Phi_\temp(G)$, where 
\begin{itemize}
\item
$\phi_0 \in \Phi_\temp(G_0)$ with a classical group $G_0$ of the same type as $G$;
\item
$\phi_\tau$ is a tempered representation of $W_F \times \SL_2(\C)$ of dimension $k$.
\end{itemize}
Let $\tau$ be the irreducible tempered representation of $\GL_k(F)$ corresponding to $\phi_\tau$.
Then for $\pi_0 \in \Pi_{\phi_0}$, 
the induced representation $\tau \rtimes \pi_0$
decomposes into a direct sum of irreducible tempered representations of $G(F)$.
The $L$-packet $\Pi_\phi$ is given by
\[
\Pi_\phi = \left\{\pi \ \middle|\ \pi \subset \tau \rtimes \pi_0,\ \pi_0 \in \Pi_{\phi_0} \right\}.
\]
Moreover there is a canonical inclusion $A_{\phi_0} \hookrightarrow A_\phi$.
If $\pi(\phi, \eta)$ is a direct summand of $\tau \rtimes \pi_0$ with $\pi_0 = \pi(\phi_0, \eta_0)$, 
then $\eta_0 = \eta|A_{\phi_0}$.

\item
Assume that 
\[
\phi = \phi_1|\cdot|^{s_1} \oplus \dots \oplus \phi_r|\cdot|^{s_r} \oplus 
\phi_0 \oplus \phi_r^\vee|\cdot|^{-s_r} \oplus \dots \oplus \phi_1^\vee|\cdot|^{-s_1},
\]
where 
\begin{itemize}
\item
$\phi_0 \in \Phi_\temp(G_0)$ with a classical group $G_0$ of the same type as $G$;
\item
$\phi_{i}$ is a tempered representation of $W_F \times \SL_2(\C)$ of dimension $k_i$ for $1 \leq i \leq r$; 
\item
$s_i$ is a real number such that $s_1 \geq \dots \geq s_r>0$.
\end{itemize}
Let $\tau_i$ be the irreducible tempered representation of $\GL_{k_i}(F)$ corresponding to $\phi_{i}$.
Then the $L$-packet $\Pi_\phi$ consists of the unique irreducible quotients $\pi$ of 
the standard modules
\[
\tau_1|\cdot|^{s_1} \times \dots \times \tau_r|\cdot|^{s_r} \rtimes \pi_0, 
\]
where $\pi_0$ runs over $\Pi_{\phi_0}$.
Moreover there is a canonical inclusion $A_{\phi_0} \hookrightarrow A_\phi$, which is in fact bijective. 
If $\pi(\phi, \eta)$ is the unique irreducible quotient of the above standard module 
with $\pi_0 = \pi(\phi_0, \eta_0)$, 
then $\eta_0 = \eta|A_{\phi_0}$. 
In this case, we denote this standard module by $I(\phi, \eta)$.
\end{enumerate}
\end{thm}

The injection $\Pi_\phi \rightarrow \widehat{A_\phi}$ is not canonical when $G = \Sp_{2n}$.
To specify this, we implicitly fix an $F^{\times2}$-orbit of non-trivial additive characters of $F$ through this paper.
\par

We have the following irreducibility criterion for standard modules.
\begin{thm}[Generalized standard module conjecture]\label{smc}
For $\phi \in \Phi(G)$, 
the standard module $I(\phi, \1)$ attached to $\pi(\phi, \1)$ is irreducible 
if and only if $\phi$ is generic.
Moreover, if $\phi$ is generic, 
then all standard modules $I(\phi, \eta)$, 
where $\eta \in \widehat{A_\phi}$ with $\eta(z_\phi) = 1$, 
are irreducible.
\end{thm}
\begin{proof}
The first assertion is the usual standard module conjecture proven in \cite{CS, Mu, HM, HO}.
The second assertion was proven by M{\oe}glin--Waldspurger \cite[Corollaire 2.14]{MW} 
for special orthogonal groups and symplectic groups.
Heiermann \cite{H} also proved the second assertion in a more general setting.
Note that their definition of generic $L$-parameters might look different from ours.
The equivalence of two definitions is called a conjecture of Gross--Prasad and Rallis, 
which was proven by Gan--Ichino \cite[Proposition B.1]{GI2}.
\end{proof}

However, even if $\phi$ is not generic, 
there might exist an irreducible standard module $I(\phi, \eta)$.
An example of such standard modules will be given by Corollary \ref{std} and Example \ref{246} below.
\par

\subsection{Extension to enhanced component groups}
To describe Jacquet modules of $\pi(\phi, \eta)$ for $\phi \in \Phi_\gp(G)$, 
it is useful to extend $\pi(\phi, \eta)$ to the case where $\eta$ is a character of the enhanced component group $\AA_\phi$
as follows.
Recall that there exists a canonical surjection $\AA_\phi \twoheadrightarrow A_\phi$
so that we have an injection $\widehat{A_\phi} \hookrightarrow \widehat{\AA_\phi}$.
For $\eta \in \widehat{\AA_\phi}$, set 
\[
\pi(\phi, \eta) = \left\{
\begin{aligned}
&\pi(\phi, \eta) \iif \eta \in \widehat{A_\phi},\ \eta(z_\phi) = 1,\\
&0 \other.
\end{aligned}
\right.
\]
In particular, $\pi(\phi, \eta)$ is irreducible or zero for any $\eta \in \widehat{\AA_\phi}$.

\subsection{M{\oe}glin's construction of tempered $L$-packets}\label{mtl}
The $L$-packets are used for local classification.
On the other hand, Arthur \cite[Theorem 2.2.1]{Ar} introduced the notion of \textbf{$A$-packets}
for global classification.
M{\oe}glin constructed the local $A$-packets in her consecutive works (e.g., \cite{Moe-06b, Moe-09}, etc.).
For a detailed why M{\oe}glin's local $A$-packets agree with Arthur's, 
one can see Xu's paper \cite{X2} in addition to the original papers of M{\oe}glin.
Since the \textbf{tempered $A$-packets} are the same notion as the tempered $L$-packets, 
M{\oe}glin's construction can be applied to the tempered $L$-packets.
\par

We explain M{\oe}glin's construction of $\Pi_\phi$ for $\phi \in \Phi_\gp(G)$.
Write 
\[
\phi = \left(\bigoplus_{i=1}^{t} \rho \boxtimes S_{a_i} \right) \oplus \phi_e
\]
with $a_1 \leq \dots \leq a_t$ and $\rho \boxtimes S_a \not\subset \phi_e$ for any $a>0$.
Take a new $L$-parameter
\[
\phi_\gg = \left(\bigoplus_{i=1}^{t} \rho \boxtimes S_{a'_i} \right) \oplus \phi_e
\]
for a bigger group $G'$ of the same type as $G$ 
such that
\begin{itemize}
\item
$a'_1 < \dots < a'_t$; 
\item
$a_i' \geq a_i$ and $a'_i \equiv a_i \bmod 2$ for any $i$; 
\end{itemize}
Then we can identify $\AA_{\phi_\gg}$ with $\AA_{\phi}$ canonically, i.e., 
$\AA_{\phi_\gg} = \AA_{\phi}$ and
if $\AA_{\phi_{\gg}} \ni \alpha \mapsto \alpha_{\rho \boxtimes S_{a_i'}} \in A_{\phi_\gg}$, 
then $\AA_{\phi} \ni \alpha \mapsto \alpha_{\rho \boxtimes S_{a_i}} \in A_\phi$, 
Let $\eta_\gg \in \widehat{\AA_{\phi_\gg}}$ be the character corresponding to $\eta \in \widehat{\AA_{\phi}}$, 
i.e., $\eta_\gg = \eta$ via $\AA_{\phi_\gg} = \AA_\phi$.

\begin{thm}[M{\oe}glin]\label{moe}
With the notation above, we have
\[
\pi(\phi, \eta) = 
\Jac_{\rho|\cdot|^{\half{a'_t-1}}, \dots, \rho|\cdot|^\half{a_t+1}}
\circ \dots \circ
\Jac_{\rho|\cdot|^{\half{a'_1-1}}, \dots, \rho|\cdot|^\half{a_1+1}}
(\pi(\phi_\gg, \eta_\gg)).
\]
\end{thm}

Using this theorem repeatedly, 
we can construct $L$-packets $\Pi_\phi$ for $\phi \in \Phi_\gp(G)$ 
from $L$-packets associated to discrete $L$-parameters for bigger groups..

\begin{ex}
We construct $\Pi_\phi$ for $\phi = S_2 \oplus S_4 \oplus S_4 \oplus S_6 \oplus S_6 \in \Phi_\gp(\SO_{23})$.
Note that $A_\phi = (\Z/2\Z) \alpha_{S_2} \oplus (\Z/2\Z) \alpha_{S_4} \oplus (\Z/2\Z) \alpha_{S_6}$ 
with $z_{\phi} = \alpha_{S_2}$.
Let $\eta \in \widehat{A_\phi}$.
We write $\eta(\alpha_{S_{2a}}) = \eta_{a} \in \{\pm1\}$.
If $\eta(z_\phi) = 1$, then $\eta_1 = +1$.
\par

Now we take new $L$-parameters
\begin{align*}
\phi_\gg &= S_2 \oplus S_4 \oplus S_6 \oplus S_8 \oplus S_{10} \in \Phi_\disc(\SO_{31}), \\
\phi'       &= S_2 \oplus S_4 \oplus S_4 \oplus S_8 \oplus S_{10} \in \Phi_\disc(\SO_{29}), \\
\phi''      &= S_2 \oplus S_4 \oplus S_4 \oplus S_6 \oplus S_{10} \in \Phi_\disc(\SO_{27}), 
\end{align*}
and we consider $\eta_\gg \in \widehat{A_{\phi_\gg}}$, $\eta' \in \widehat{A_{\phi'}}$ 
and $\eta'' \in \widehat{A_{\phi''}}$ given by
\begin{itemize}
\item
$\eta_\gg(\alpha_{S_2}) = \eta'(\alpha_{S_2}) = \eta''(\alpha_{S_2}) = \eta_1 = +1$; 
\item
$\eta_\gg(\alpha_{S_4}) = \eta_\gg(\alpha_{S_6}) = \eta'(\alpha_{S_4}) = \eta''(\alpha_{S_4}) = \eta_2$; 
\item
$\eta_\gg(\alpha_{S_8}) = \eta_\gg(\alpha_{S_{10}}) = \eta'(\alpha_{S_8}) = \eta'(\alpha_{S_{10}}) 
= \eta''(\alpha_{S_6}) = \eta''(\alpha_{S_{10}}) = \eta_3$.
\end{itemize}
Then Theorem \ref{moe} says that 
\begin{align*}
\Jac_{|\cdot|^{\half{5}}} (\pi(\phi_\gg, \eta_\gg)) &= \pi(\phi', \eta'), \\
\Jac_{|\cdot|^{\half{7}}} (\pi(\phi', \eta')) &= \pi(\phi'', \eta''), \\
\Jac_{|\cdot|^{\half{9}}, |\cdot|^{\half{7}}} (\pi(\phi'', \eta'')) &= \pi(\phi, \eta).
\end{align*}
\end{ex}

\section{Description of Jacquet modules}\label{desc}
In this section, we state the main theorems, 
which compute 
the semisimplifications of the Jacquet modules of $\pi(\phi, \eta)$ for $\phi \in \Phi_\gp(G)$.

\subsection{Statements}
Note that: 
\begin{lem}\label{nonvanish}
For $\phi \in \Phi_\gp(G)$ and $x \in \R$, 
if $\Jac_{\rho|\cdot|^x}(\pi) \not= 0$ for some $\pi \in \Pi_\phi$, 
then $x$ is a non-negative half-integer and $\rho \boxtimes S_{2x+1} \subset \phi$.
\end{lem}
\begin{proof}
When $\phi \in \Phi_\disc(G)$, it follows from \cite[Lemma 7.3]{X1}.
We may assume that $\phi \in \Phi_\gp(G) \setminus \Phi_\disc(G)$. 
Then there exists an irreducible representation $\rho' \boxtimes S_{a}$ 
which $\phi$ contains at least multiplicity two.
By Theorem \ref{LLC} (4), we have 
\[
\pi \hookrightarrow \St(\rho', a) \rtimes \pi_0
\]
for some $\pi_0 \in \Pi_{\phi_0}$ with $\phi_0 = \phi - (\rho' \boxtimes S_{a})^{\oplus2}$.
Then by Corollary \ref{cor.tdc}, we have
\begin{align*}
\Jac_{\rho|\cdot|^{x}}(\pi) 
&\hookrightarrow 
(1 \otimes \St(\rho', a)) \rtimes \Jac_{\rho|\cdot|^{x}}(\pi_0)
\\&+
\left\{
\begin{aligned}
&2\pair{\rho; x,\dots,-(x-1)} \rtimes \pi_0 \iif \rho' \cong \rho, a = 2x+1, \\
&0 \other.
\end{aligned}
\right.
\end{align*}
Hence if $\Jac_{\rho|\cdot|^{x}}(\pi) \not= 0$, 
then $\Jac_{\rho|\cdot|^{x}}(\pi_0) \not= 0$ or
$\rho \boxtimes S_{2x+1} \cong \rho' \boxtimes S_{a} \subset \phi$.
By induction, we conclude that $\rho \boxtimes S_{2x+1} \subset \phi$ as desired.
\end{proof}

The following is the first main theorem, 
which is a description of $\Jac_{\rho|\cdot|^x}(\pi(\phi, \eta))$.

\begin{thm}\label{jac1}
Let $\phi \in \Phi_\gp(G)$ and $\eta \in \widehat{\AA_\phi}$ such that $\pi(\phi, \eta) \not= 0$. 
Fix a non-negative half-integer $x \in (1/2)\Z$.
Write 
\[
\phi = \phi_0 \oplus (\rho \boxtimes S_{2x+1})^{\oplus m}
\] 
with $\rho \boxtimes S_{2x+1} \not \subset \phi_0$ and $m > 0$.

\begin{enumerate}
\item
Assume that $m \geq 3$.
Take $\delta \in \{1,2\}$ such that $\delta \equiv m \bmod 2$.
Then 
\begin{align*}
&\Jac_{\rho|\cdot|^x} (\pi(\phi, \eta)) 
\\&= 
(m-\delta) \cdot \pair{\rho; \underbrace{x, x-1, \dots,-(x-1)}_{2x}} 
\rtimes \pi \left(\phi - (\rho \boxtimes S_{2x+1})^{\oplus 2}, \eta \right)
\\&+
\underbrace{\St(\rho, 2x+1) \times \dots \times \St(\rho, 2x+1)}_{(m-\delta)/2} 
\rtimes \Jac_{\rho|\cdot|^{x}} \left(\pi \left(\phi_0 \oplus (\rho \boxtimes S_{2x+1})^{\oplus \delta}, \eta \right)\right).
\end{align*}
Here, we canonically identify the (usual) component groups of $\phi - (\rho \boxtimes S_{2x+1})^{\oplus 2}$ 
and $\phi_0 \oplus (\rho \boxtimes S_{2x+1})^{\oplus \delta}$
with $A_\phi$, so that we regard $\eta$ as a character of these groups.

\item
Assume that $x > 0$ and $m = 1$. 
Set 
\[
\phi' = \phi - (\rho \boxtimes S_{2x+1}) \oplus (\rho \boxtimes S_{2x-1}).
\] 
There is a canonical inclusion $\AA_{\phi'} \hookrightarrow \AA_\phi$, 
which is in fact bijective if $x > 1/2$.
Let $\eta' \in \widehat{\AA_{\phi'}}$ be the character corresponding to $\eta \in \widehat{\AA_\phi}$, 
i.e., $\eta' = \eta|\AA_{\phi'}$.
Then 
\begin{align*}
\Jac_{\rho|\cdot|^x} (\pi(\phi, \eta)) = \pi(\phi', \eta').
\end{align*}
In particular, $\Jac_{\rho|\cdot|^x} (\pi(\phi, \eta))$ is irreducible or zero.

\item
Assume that $x > 0$ and $m = 2$. 
Set $\eta_+ = \eta$, and
take the unique character $\eta_- \in \widehat{\AA_\phi}$ 
so that $\eta_-|\AA_{\phi_0} = \eta_+|\AA_{\phi_0}$, $\eta_-(z_\phi) = \eta_+(z_\phi) = 1$ 
but $\eta_- \not= \eta_+$.
For $\phi'$ as in (3) and for $\epsilon \in \{\pm1\}$, 
let $\eta'_\epsilon \in \widehat{\AA_{\phi'}}$ be the character corresponding to $\eta_\epsilon \in \widehat{\AA_\phi}$
via the canonical inclusion $\AA_{\phi'} \hookrightarrow \AA_\phi$.
Then 
\begin{align*}
\Jac_{\rho|\cdot|^x} (\pi(\phi, \eta)) = 
\pair{\rho; x, x-1, \dots, -(x-1)} \rtimes \pi(\phi_0, \eta|\AA_{\phi_0})
+ \pi(\phi', \eta'_+) - \pi(\phi', \eta'_-).
\end{align*}

\item
Assume that $x = 0$.
If $m=1$, then $\Jac_{\rho}(\pi(\phi, \eta)) = 0$.
If $m=2$, then $\Jac_{\rho}(\pi(\phi, \eta)) = \pi(\phi_0, \eta|\AA_{\phi_0})$.

\end{enumerate}
\end{thm}

When $\phi \in \Phi_\disc(G)$, Theorem \ref{jac1} has been already proven by Xu (\cite[Lemma 7.3]{X1}).
In (2) (\resp (3)), we note that
$\pi(\phi', \eta')$ (\resp $\pi(\phi', \eta'_\epsilon)$) can be zero even if $\pi(\phi, \eta) \not= 0$.
In (3), the character $\eta_-$ is characterized so that 
\[
\pi(\phi, \eta_+) \oplus \pi(\phi, \eta_-) = \St(\rho, 2x+1) \rtimes \pi(\phi_0, \eta|\AA_{\phi_0}).
\]
\par

The second main theorem concerns $\mu^*_\rho(\pi)$.

\begin{thm}\label{jac2}
Let $\phi \in \Phi_\gp(G)$, 
and write
\[
\phi = \left(\bigoplus_{i=1}^{t} \rho \boxtimes S_{a_i} \right) \oplus \phi_e
\]
with $a_1 \leq \dots \leq a_t$ and $\rho \boxtimes S_a \not\subset \phi_e$ for any $a>0$, 
and $a_i = 2x_i+1$.
For $0 \leq m \leq (2d)^{-1} \cdot \dim(\phi)$, 
we denote by $K_\phi^{(m)}$ the set of tuples of integers $\ub{k} = (k_1, \dots, k_t)$ such that 
\begin{itemize}
\item
$0 \leq k_i \leq a_i$ for any $i$; 
\item
$k_{i-1} \geq k_{i}$ if $a_{i-1} = a_i$; 
\item
$k_1+\dots+k_t = m$.
\end{itemize}
For $\ub{k} \in K_\phi^{(m)}$, 
set 
\[
\ub{x}(\ub{k}) 
= (\underbrace{x_1, \dots, x_1-k_1+1}_{k_1}, \dots, \underbrace{x_t, \dots, x_t-k_t+1}_{k_t})
\in \Omega_{m}.
\]
For $\ub{k}, \ub{l} \in K_\phi$, 
we set 
\[
m_{\ub{k}, \ub{l}} = \dim_\C \Jac_{\rho|\cdot|^{\ub{x}(\ub{k})}}(\Delta_{\ub{x}(\ub{l})}), 
\]
and define $(m'_{\ub{k}, \ub{l}})_{\ub{k}, \ub{l} \in K_\phi^{(m)}}$ 
to be the inverse matrix of $(m_{\ub{k}, \ub{l}})_{\ub{k}, \ub{l} \in K_\phi^{(m)}}$, i.e., 
\[
\sum_{\ub{k'} \in K_\phi^{(m)}} m'_{\ub{k}, \ub{k'}} m_{\ub{k'}, \ub{l}} = 
\left\{
\begin{aligned}
&1 \iif \ub{k} = \ub{l}, \\
&0 \iif \ub{k} \not= \ub{l}.
\end{aligned}
\right.
\]
Then for $\pi \in \Pi_\phi$, we have
\begin{align*}
\mu^*_\rho(\pi) 
&= \sum_{m=0}^{(2d)^{-1}\dim(\phi)}
\sum_{\ub{k}, \ub{l} \in K_\phi^{(m)}} 
m'_{\ub{k}, \ub{l}} \cdot
\Delta_{\ub{x}(\ub{k})}
\otimes 
\Jac_{\rho|\cdot|^{\ub{x}(\ub{l})}}(\pi). 
\end{align*}
\end{thm}


When we formally regard $(\Delta_{\ub{x}(\ub{k})})_{\ub{k} \in K_\phi^{(m)}}$ 
and $(\otimes \Jac_{\rho|\cdot|^{\ub{x}(\ub{l})}}(\pi))_{\ub{l} \in K_\phi^{(m)}}$ as column vectors, 
we have
\[
\sum_{\ub{k}, \ub{l} \in K_\phi^{(m)}} 
m'_{\ub{k}, \ub{l}} \cdot
\Delta_{\ub{x}(\ub{k})}
\otimes 
\Jac_{\rho|\cdot|^{\ub{x}(\ub{l})}}(\pi) 
= {}^t (\Delta_{\ub{x}(\ub{k})}) \cdot (m_{\ub{k}, \ub{l}})^{-1}
\cdot (\otimes \Jac_{\rho|\cdot|^{\ub{x}(\ub{l})}}(\pi)).
\]
By Lemma \ref{dim}, $(m_{\ub{k}, \ub{l}})_{\ub{k}, \ub{l} \in K_\phi^{(m)}}$ is a ``triangular matrix'', 
which can be computed inductively.
Here, we regard $K_\phi^{(m)}$ as a totally ordered set with respect to the lexicographical order.
The diagonal entries $m_{\ub{k}, \ub{k}}$ are given in Lemma \ref{dim} explicitly.
\par

By Tadi{\'c}'s formula (Theorem \ref{tdc}), Lemma \ref{circ}, and Theorems \ref{jac1}, \ref{jac2}, 
we can deduce the following corollary.
\begin{cor}
We can compute $\mu^*(\pi)$ explicitly for any $\pi \in \Pi_\phi$ with $\phi \in \Phi_\gen(G)$.
\end{cor}

\subsection{Examples}
We shall give some examples.

\begin{ex}\label{ab}
Fix two positive integers $a,b$ such that $a \equiv b \mod2$ and $a < b$, 
and consider 
\[
\phi = \rho \boxtimes (S_a \oplus S_b) \in \Phi_\disc(\SO_{d(a+b)+1}).
\]
Then $\Pi_\phi = \{\pi_+(a,b), \pi_-(a,b)\}$ with generic $\pi_+(a,b)$ and non-generic $\pi_-(a,b)$.
Note that both $\pi_+(a,b)$ and $\pi_-(a,b)$ are discrete series.
We compute $\mu^*(\pi_\epsilon(a,b))$ for $\epsilon \in \{\pm1\}$.
\par

Note that $K_\phi^{(m)} = \{(k_1,k_2) \in \Z^2\ |\ 0 \leq k_1 \leq a,\ 0 \leq k_2 \leq b, \ k_1+k_2 = m\}$.
For $(k_1,k_2) \in K_\phi^{(m)}$, 
\[
\Delta_{\ub{x}(k_1,k_1)} 
= \pair{\rho; \half{a-1}, \dots, \half{a+1}-k_1} \times \pair{\rho; \half{b-1}, \dots, \half{b+1}-k_2}.
\]
This induced representation is irreducible unless $(a+3)/2-k_1 \leq (b+1)/2-k_2 \leq (a+1)/2$, 
i.e., $(b-a)/2 \leq k_2 \leq (b-a)/2 + k_1-1$.
Moreover, one can easy to see that for $(l_1, l_2)$, the virtual representation 
\[
\sum_{(k_1,k_2) \in K_\phi^{(m)}} m'_{(k_1,k_2), (l_1, l_2)} \cdot \Delta_{\ub{x}(k_1,k_2)}
\]
is the unique irreducible subrepresentation $\tau_{\ub{x}(l_1,l_2)}$ of $\Delta_{\ub{x}(l_1,l_2)}$
(cf.~see \cite[Proposition 4.6]{Z}).
\par

Note that 
\[
\Jac_{\rho|\cdot|^{\half{a-1}}, \dots, \rho|\cdot|^{\half{a+1}-k_1}}(\pi_{\pm}(a,b)) 
= \left\{
\begin{aligned}
&\pi_{\epsilon}(a-2k_1, b) \iif k_1 \leq a/2, \\
&0 \other.
\end{aligned}
\right.
\] 
Here, when $k_1 = a/2$, 
we understand that $\pi_+(0,b)$ is the unique element in $\Pi_{\rho \boxtimes S_b}$, 
and $\pi_-(0,b) = 0$.
Moreover, for $(k_1, k_2) \in K_\phi^{(m)}$, 
when $k_1 \leq a/2$ and $k_2 \leq (b-a)/2+k_1$, we have
\[
\Jac_{\rho|\cdot|^{\ub{x}(k_1,k_2)}}(\pi_\epsilon(a,b)) = \pi_\epsilon(a-2k_1, b-2k_2).
\]
When $k_1 \leq a/2$ and $(b-a)/2+k_1+1 \leq k_2 \leq b/2$, we have
\begin{align*}
\Jac_{\rho|\cdot|^{\ub{x}(k_1,k_2)}}(\pi_\epsilon(a,b)) = 
&\pair{\rho; \half{a-1}-k_1, \dots, -\half{b-1}+k_2} \rtimes \1_{\SO_1(F)}
\\&+ \pi_\epsilon(b-2k_2, a-2k_1) - \pi_{-\epsilon}(b-2k_2, a-2k_1).
\end{align*}
When $k_1 \leq a/2$ and $b/2 < k_2 \leq (a+b)/2-k_1$, we have
\begin{align*}
\Jac_{\rho|\cdot|^{\ub{x}(k_1,k_2)}}(\pi_\epsilon(a,b)) = 
&\pair{\rho; \half{a-1}-k_1, \dots, -\half{b-1}+k_2} \rtimes \1_{\SO_1(F)}.
\end{align*}
In particular, if $k_1+k_2 = (a+b)/2$, 
then $\Jac_{\rho|\cdot|^{\ub{x}(k_1,k_2)}}(\pi_\epsilon(a,b)) = \1_{\SO_1(F)}$.
Hence $\semi\Jac_{P_{d(a+b)/2}}(\pi_\epsilon(a,b))$ contains the irreducible representation
\begin{align*}
&\pair{\rho; \half{a-1}, \dots, \half{a+1}-k_1} \times \pair{\rho; \half{b-1}, \dots, -\left(\half{a-1}-k_1\right)}
\otimes \1_{\SO_1(F)}
\end{align*}
with multiplicity one
if $k_1 < a/2$, or if $k_1 = a/2$ and $\epsilon = +1$.
\end{ex}

\begin{ex}\label{244}
Consider the $L$-parameter $\phi = S_2 \oplus S_4 \oplus S_4 \in \Phi_\gp(\SO_{11})$.
Then $\Pi_\phi$ has two elements $\pi_+(2,4,4)$ and $\pi_-(2,4,4)$ 
with generic $\pi_+(2,4,4)$ and non-generic $\pi_-(2,4,4)$.
Then 
\[
K^{(m)} = \{(k_1,k_2,k_3) \in \Z^3\ |\ 0 \leq k_1 \leq 2,\ 0 \leq k_3 \leq k_2 \leq 4, k_1+k_2+k_3 = m\}
\]
for $0 \leq m \leq 5$.
Write $\Pi_{\ub{x}(\ub{k})}^{\epsilon} = \Jac_{\rho|\cdot|^{\ub{x}(\ub{k})}}(\pi_\epsilon(2,4,4))$ 
for $\epsilon \in \{\pm1\}$, 
and $\St_a = \St(\1_{\GL_1(F)}, a)$.
We denote by ${\det}_a$ the determinant character of $\GL_a(F)$.

\begin{enumerate}
\item
When $m = 1$, we have $K_\phi^{(1)} = \{(1,0,0) > (0,1,0)\}$.
Since 
\[
\begin{pmatrix}
\Delta_{\ub{x}(1,0,0)} &
\Delta_{\ub{x}(0,1,0)}
\end{pmatrix}
= 
\begin{pmatrix}
|\cdot|^{\half{1}} & |\cdot|^{\half3}
\end{pmatrix},
\]
we have
\[
(m_{\ub{k}, \ub{k'}})_{\ub{k}, \ub{k'}}^{-1} = 
\begin{pmatrix}
1 & 0 \\ 0 & 1
\end{pmatrix}^{-1}
=
\begin{pmatrix}
1 & 0 \\ 0 & 1
\end{pmatrix}.
\]
Moreover
\[
\begin{pmatrix}
\Pi_{\ub{x}(1,0,0)}^\epsilon \\
\Pi_{\ub{x}(0,1,0)}^\epsilon
\end{pmatrix}
= 
\begin{pmatrix}
\pi_\epsilon(4,4) \\
|\cdot|^{\half{1}}\St_3 \rtimes \pi_+(2) + \epsilon \cdot \pi_+(2,2,4)
\end{pmatrix}.
\]
Hence
\begin{align*}
\semi\Jac_{P_1}(\pi_\epsilon(2,4,4)) 
&= 
|\cdot|^{\half{1}} \otimes \pi_\epsilon(4,4)
\\&+
|\cdot|^{\half{3}} \otimes 
\left(|\cdot|^{\half{1}}\St_3 \rtimes \pi_+(2) + \epsilon \cdot \pi_+(2,2,4)\right).
\end{align*}

\item
When $m = 2$, we have $K_\phi^{(2)} = \{(2,0,0) > (1,1,0) > (0,2,0) > (0,1,1)\}$.
Since 
\begin{align*}
&\begin{pmatrix}
\Delta_{\ub{x}(2,0,0)} &
\Delta_{\ub{x}(1,1,0)} &
\Delta_{\ub{x}(0,2,0)} &
\Delta_{\ub{x}(0,1,1)}
\end{pmatrix}
\\&= 
\begin{pmatrix}
\St_2 & |\cdot|^{\half{1}} \times |\cdot|^{\half{3}} & |\cdot|^1 \St_2 & |\cdot|^{\half{3}} \times |\cdot|^{\half{3}}
\end{pmatrix},
\end{align*}
we have
\[
(m_{\ub{k}, \ub{k'}})_{\ub{k}, \ub{k'}}^{-1} = 
\begin{pmatrix}
1 & 0 & 0 & 0 \\ 
0 & 1 & 0 & 0 \\
0 & 1 & 1 & 0 \\
0 & 0 & 0 & 2
\end{pmatrix}^{-1}
=
\begin{pmatrix}
1 & 0 & 0 & 0 \\ 
0 & 1 & 0 & 0 \\
0 & -1 & 1 & 0 \\
0 & 0 & 0 & \half{1}
\end{pmatrix}.
\]
Moreover
\[
\begin{pmatrix}
\Pi_{\ub{x}(2,0,0)}^\epsilon \\
\Pi_{\ub{x}(1,1,0)}^\epsilon \\
\Pi_{\ub{x}(0,2,0)}^\epsilon \\
\\
\Pi_{\ub{x}(0,1,1)}^\epsilon 
\end{pmatrix}
= 
\begin{pmatrix}
0 \\
|\cdot|^{\half{1}} \St_3 \rtimes \1_{\SO_1(F)} + \pi_\epsilon(2,4) - \pi_{-\epsilon}(2,4) \\
|\cdot|^1 \St_2 \rtimes \pi_+(2) + |\cdot|^{\half{1}} \St_3 \rtimes \1_{\SO_1(F)} \\
+ \epsilon \left( |\cdot|^{\half{1}} \rtimes \pi_+(4) + \pi_+(2,4) \right)\\
(1+\epsilon) \cdot \pi_+(2,2,2)
\end{pmatrix}.
\]
Hence
\begin{align*}
\semi\Jac_{P_2}(\pi_\epsilon(2,4,4)) 
&= 
|{\det}_2|^{\half{1}} \otimes 
\left(
|\cdot|^{\half{1}} \St_3 \rtimes \1_{\SO_1(F)} + \pi_\epsilon(2,4) - \pi_{-\epsilon}(2,4)
\right)
\\&+
|\cdot|^{\half{1}}\St_2 \otimes 
\Big(
|\cdot|^1 \St_2 \rtimes \pi_+(2) + |\cdot|^{\half{1}} \St_3 \rtimes \1_{\SO_1(F)} 
\\&
\phantom{|\cdot|^{\half{1}}\St_2 \otimes \otimes}
+ 
\epsilon \left( |\cdot|^{\half{1}} \rtimes \pi_+(4) + \pi_+(2,4) \right)
\Big)\\
\\&+ 
\left( |\cdot|^{\half{3}} \times |\cdot|^{\half{3}} \right) \otimes
\half{1+\epsilon} \cdot \pi_+(2,2,2).
\end{align*}

\item
When $m = 3$, we have $K_\phi^{(3)} = \{(2,1,0) > (1,2,0) > (1,1,1) > (0,3,0) > (0,2,1)\}$.
Since 
\begin{align*}
&\begin{pmatrix}
\Delta_{\ub{x}(2,1,0)} &
\Delta_{\ub{x}(1,2,0)} &
\Delta_{\ub{x}(1,1,1)} &
\Delta_{\ub{x}(0,3,0)} &
\Delta_{\ub{x}(0,2,1)}
\end{pmatrix}
\\&= 
\begin{pmatrix}
\St_2 \times |\cdot|^{\half{3}} & 
|\cdot|^{\half{1}} \times |\cdot|^{1}\St_2 & 
|\cdot|^{\half{1}} \times |\cdot|^{\half{3}} \times |\cdot|^{\half{3}} &
|\cdot|^{\half{1}} \St_3 & 
|\cdot|^{1}\St_2 \times |\cdot|^{\half{3}}
\end{pmatrix},
\end{align*}
we have
\[
(m_{\ub{k}, \ub{k'}})_{\ub{k}, \ub{k'}}^{-1} = 
\begin{pmatrix}
1 & 0 & 0 & 0 & 0 \\ 
0 & 1 & 0 & 0 & 0\\
0 & 0 & 2 & 0 & 0\\
1 & 0 & 0 & 1 & 0 \\
0 & 0 & 2 & 0 & 1
\end{pmatrix}^{-1}
=
\begin{pmatrix}
1 & 0 & 0 & 0 & 0 \\ 
0 & 1 & 0 & 0 & 0\\
0 & 0 & \half{1} & 0 & 0\\
-1 & 0 & 0 & 1 & 0 \\
0 & 0 & -1 & 0 & 1\end{pmatrix}.
\]
Moreover
\[
\begin{pmatrix}
\Pi_{\ub{x}(2,1,0)}^\epsilon \\
\Pi_{\ub{x}(1,2,0)}^\epsilon \\
\Pi_{\ub{x}(1,1,1)}^\epsilon \\
\Pi_{\ub{x}(0,3,0)}^\epsilon \\
\Pi_{\ub{x}(0,2,1)}^\epsilon 
\end{pmatrix}
= 
\begin{pmatrix}
0 \\
|\cdot|^1 \St_2 \rtimes \1_{\SO_1(F)} + \epsilon \cdot \pi_+(4) \\
2 \cdot \pi_\epsilon(2,2) \\
|\cdot|^{\half{3}} \rtimes \pi_+(2) + \epsilon\cdot \pi_+(4) \\
(1+\epsilon) \left( |\cdot|^{\half{1}} \rtimes \pi_+(2) + \pi_+(2,2) \right) + \pi_-(2,2)
\end{pmatrix}.
\]
Hence
\begin{align*}
\semi\Jac_{P_3}(\pi_\epsilon(2,4,4)) 
&= 
\left( |\cdot|^{\half{1}} \times |\cdot|^{1}\St_2 \right) \otimes 
\left(
|\cdot|^1 \St_2 \rtimes \1_{\SO_1(F)} + \epsilon \cdot \pi_+(4)
\right)
\\&+
\left(|{\det}_2|^{1} \times |\cdot|^{\half{3}} \right) \otimes 
\pi_\epsilon(2,2)
\\&+ 
|\cdot|^{\half{1}} \St_3 \otimes 
\left(
|\cdot|^{\half{3}} \rtimes \pi_+(2) + \epsilon\cdot \pi_+(4) 
\right)
\\&+
\left( |\cdot|^{1}\St_2 \times |\cdot|^{\half{3}} \right) \otimes
\left(
(1+\epsilon) |\cdot|^{\half{1}} \rtimes \pi_+(2) 
+\pi_+(2,2)
+ \epsilon \cdot \pi_{-\epsilon}(2,2)
\right).
\end{align*}

\item
When $m = 4$, we have $K_\phi^{(4)} = \{(2,2,0) > (2,1,1) > (1,3,0) > (1,2,1) > (0,4,0) > (0,3,1) > (0,2,2)\}$.
Since 
\begin{align*}
&\begin{pmatrix}
\Delta_{\ub{x}(2,2,0)} \\
\Delta_{\ub{x}(2,1,1)} \\
\Delta_{\ub{x}(1,3,0)} \\
\Delta_{\ub{x}(1,2,1)} \\
\Delta_{\ub{x}(0,4,0)} \\
\Delta_{\ub{x}(0,3,1)} \\
\Delta_{\ub{x}(0,2,2)} 
\end{pmatrix}
= 
\begin{pmatrix}
\St_2 \times |\cdot|^1 \St_2 \\
\St_2 \times |\cdot|^{\half{3}} \times |\cdot|^{\half{3}} \\
|\cdot|^{\half{1}} \times |\cdot|^{\half{1}} \St_3 \\
|\cdot|^{\half{1}} \times |\cdot|^{1} \St_2 \times |\cdot|^{\half{3}} \\
\St_4 \\
|\cdot|^{\half{1}}\St_3 \times |\cdot|^{\half{3}} \\
|\cdot|^1 \St_2 \times |\cdot|^1 \St_2 
\end{pmatrix},
\end{align*}
we have
\[
(m_{\ub{k}, \ub{k'}})_{\ub{k}, \ub{k'}}^{-1} = 
\begin{pmatrix}
1 & 0 & 0 & 0 & 0 & 0 & 0 \\ 
0 & 2 & 0 & 0 & 0 & 0 & 0 \\ 
1 & 0 & 1 & 0 & 0 & 0 & 0 \\ 
0 & 0 & 0 & 1 & 0 & 0 & 0 \\ 
0 & 0 & 0 & 0 & 1 & 0 & 0 \\ 
0 & 2 & 0 & 0 & 0 & 1 & 0 \\ 
0 & 0 & 0 & 3 & 0 & 0 & 2
\end{pmatrix}^{-1}
=
\begin{pmatrix}
1 & 0 & 0 & 0 & 0 & 0 & 0 \\ 
0 & \half{1} & 0 & 0 & 0 & 0 & 0 \\ 
-1 & 0 & 1 & 0 & 0 & 0 & 0 \\ 
0 & 0 & 0 & 1 & 0 & 0 & 0 \\ 
0 & 0 & 0 & 0 & 1 & 0 & 0 \\ 
0 & -1 & 0 & 0 & 0 & 1 & 0 \\ 
0 & 0 & 0 & -\half{3} & 0 & 0 & \half{1}
\end{pmatrix}.
\]
Moreover
\[
\begin{pmatrix}
\Pi_{\ub{x}(2,2,0)}^\epsilon \\
\Pi_{\ub{x}(2,1,1)}^\epsilon \\
\Pi_{\ub{x}(1,3,0)}^\epsilon \\
\Pi_{\ub{x}(1,2,1)}^\epsilon \\
\Pi_{\ub{x}(0,4,0)}^\epsilon \\
\Pi_{\ub{x}(0,3,1)}^\epsilon \\
\Pi_{\ub{x}(0,2,2)}^\epsilon 
\end{pmatrix}
= 
\begin{pmatrix}
0 \\
0 \\
|\cdot|^{\half{3}} \rtimes \1_{\SO_1(F)} \\
|\cdot|^{\half{1}} \rtimes \1_{\SO_1(F)} + \epsilon \cdot \pi_+(2) \\
\pi_+(2) \\
(1+\epsilon) \cdot \pi_+(2) \\
(3+2\epsilon) |\cdot|^{\half{1}} \rtimes \1_{\SO_1(F)} + (1+2\epsilon) \cdot \pi_+(2)
\end{pmatrix}.
\]
Hence
\begin{align*}
\semi\Jac_{P_4}(\pi_\epsilon(2,4,4)) 
&= 
\left( |\cdot|^{\half{1}} \times |\cdot|^{\half{1}} \St_3 \right) \otimes |\cdot|^{\half{3}} \rtimes \1_{\SO_1(F)}
\\&+
\left( |\cdot|^1 \St_2 \times |{\det}_2|^1 \right) \otimes 
\left(
|\cdot|^{\half{1}} \rtimes \1_{\SO_1(F)} + \epsilon \cdot \pi_+(2)
\right)
\\&+
\St_4 \otimes \pi_+(2)
\\&+ 
\left( |\cdot|^{\half{1}}\St_3 \times |\cdot|^{\half{3}} \right) \otimes
(1+\epsilon) \cdot \pi_+(2)
\\&+ 
\left( |\cdot|^1 \St_2 \times |\cdot|^1 \St_2 \right) \otimes
\left(
(1+\epsilon) |\cdot|^{\half{1}} \rtimes \1_{\SO_1(F)} + \half{1+\epsilon} \cdot \pi_+(2)
\right).
\end{align*}

\item
When $m = 5$, we have $K_\phi^{(5)} = \{(2,3,0) > (2,2,1) > (1,4,0) > (1,3,1) > (1,2,2) > (0,4,1) > (0,3,2) \}$.
Since 
\begin{align*}
&\begin{pmatrix}
\Delta_{\ub{x}(2,3,0)} \\
\Delta_{\ub{x}(2,2,1)} \\
\Delta_{\ub{x}(1,4,0)} \\
\Delta_{\ub{x}(1,3,1)} \\
\Delta_{\ub{x}(1,2,2)} \\
\Delta_{\ub{x}(0,4,1)} \\
\Delta_{\ub{x}(0,3,2)} 
\end{pmatrix}
= 
\begin{pmatrix}
\St_2 \times |\cdot|^{\half{1}} \St_3 \\
\St_2 \times |\cdot|^{1}\St_2 \times |\cdot|^{\half{3}} \\
|\cdot|^{\half{1}} \times \St_4 \\
|\cdot|^{\half{1}} \times |\cdot|^{\half{1}} \St_3 \times |\cdot|^{\half{3}} \\
|\cdot|^{\half{1}} \times |\cdot|^1 \St_2 \times |\cdot|^1 \St_2 \\
\St_4 \times |\cdot|^{\half{3}} \\
|\cdot|^{\half{1}}\St_3 \times |\cdot|^{1}\St_2 
\end{pmatrix},
\end{align*}
we have
\[
(m_{\ub{k}, \ub{k'}})_{\ub{k}, \ub{k'}}^{-1} = 
\begin{pmatrix}
1 & 0 & 0 & 0 & 0 & 0 & 0 \\ 
0 & 1 & 0 & 0 & 0 & 0 & 0 \\ 
0 & 0 & 1 & 0 & 0 & 0 & 0 \\ 
0 & 1 & 0 & 1 & 0 & 0 & 0 \\ 
0 & 0 & 0 & 0 & 2 & 0 & 0 \\ 
0 & 0 & 0 & 0 & 0 & 1 & 0 \\ 
0 & 2 & 0 & 1 & 0 & 0 & 1
\end{pmatrix}^{-1}
=
\begin{pmatrix}
1 & 0 & 0 & 0 & 0 & 0 & 0 \\ 
0 & 1 & 0 & 0 & 0 & 0 & 0 \\ 
0 & 0 & 1 & 0 & 0 & 0 & 0 \\ 
0 & -1 & 0 & 1 & 0 & 0 & 0 \\ 
0 & 0 & 0 & 0 & \half{1} & 0 & 0 \\ 
0 & 0 & 0 & 0 & 0 & 1 & 0 \\ 
0 & -1 & 0 & -1 & 0 & 0 & 1
\end{pmatrix}.
\]
Moreover
\[
\begin{pmatrix}
\Pi_{\ub{x}(2,3,0)}^\epsilon \\
\Pi_{\ub{x}(2,2,1)}^\epsilon \\
\Pi_{\ub{x}(1,4,0)}^\epsilon \\
\Pi_{\ub{x}(1,3,1)}^\epsilon \\
\Pi_{\ub{x}(1,2,2)}^\epsilon \\
\Pi_{\ub{x}(0,4,1)}^\epsilon \\
\Pi_{\ub{x}(0,3,2)}^\epsilon 
\end{pmatrix}
= 
\begin{pmatrix}
0 \\
0 \\
\1_{\SO_1(F)} \\
\1_{\SO_1(F)} \\
(1+\epsilon) \cdot \1_{\SO_1(F)} \\
0 \\
(1+\epsilon) \cdot \1_{\SO_1(F)}
\end{pmatrix}.
\]
Hence
\begin{align*}
\semi\Jac_{P_5}(\pi_\epsilon(2,4,4)) 
&= 
\left( |\cdot|^{\half{1}} \times \St_4  \right) \otimes \1_{\SO_1(F)}
\\&+
\left(|\cdot|^{\half{1}} \St_3 \times |{\det}_2|^1 \right) \otimes \1_{\SO_1(F)}
\\&+
\left( |\cdot|^{\half{1}} \times |\cdot|^1 \St_2 \times |\cdot|^1 \St_2 \right) \otimes
\half{1+\epsilon} \cdot \1_{\SO_1(F)}
\\&+
\left( |\cdot|^{\half{1}}\St_3 \times |\cdot|^{1}\St_2 \right) \otimes
(1+\epsilon) \cdot \1_{\SO_1(F)}.
\end{align*}

\end{enumerate}
\end{ex}

\if()
\begin{ex}\label{3333}
Consider the $L$-parameter $\phi = (\rho \boxtimes S_3)^{\oplus 4} \in \Phi_\gp(\SO_{12d+1})$.
Then $\Pi_\phi$ has two elements $\pi_+$ and $\pi_-$, 
and 
\[
K_\phi^{(m)} = \{\ub{k} = (k_1, k_2, k_3, k_4)\ |\ 0 \leq k_4 \leq \dots \leq k_1 \leq 3,\ k_1+\dots+k_4 = m\}.
\]
Write $\Pi_{\ub{x}(\ub{k})}^{\epsilon} = \Jac_{\rho|\cdot|^{\ub{x}(\ub{k})}}(\pi_\epsilon)$ 
for $\epsilon \in \{\pm1\}$.
\begin{enumerate}
\setcounter{enumi}{-1}
\item 
When $m = 0$, we have $K_\phi^{(0)} = \{(0,0,0,0)\}$ and $(m_{\ub{k}, \ub{k'}}) = (1)$.
Moreover: 
\begin{itemize}
\item
If $\ub{k} = (0,0,0,0)$, 
then $\Delta_{\ub{x}(0,0,0,0)} = \1_{\GL_0(F)}$ and 
\[
\Pi_{\ub{x}(0,0,0,0)}^{\epsilon} = \pi_\epsilon.
\]
\end{itemize}
Hence
\[
\semi\Jac_{P_0}(\pi_\epsilon) = \1_{\GL_0(F)} \otimes \pi_\epsilon.
\]

\item
When $m = 1$, we have $K_\phi^{(1)} = \{(1,0,0,0)\}$ and $(m_{\ub{k}, \ub{k'}}) = (1)$.
Moreover: 
\begin{itemize}
\item
If $\ub{k} = (1,0,0,0)$, 
then $\Delta_{\ub{x}(1,0,0,0)} = \rho|\cdot|^1$ and 
\begin{align*}
\Pi_{\ub{x}(1,0,0,0)}^{\epsilon} 
= 
&3 \cdot \pair{\rho; 1,0} \rtimes \pi_{\epsilon}(3,3) + \pair{\rho; 1,0} \rtimes \pi_{-\epsilon}(3,3)
\\&+ 
\St(\rho, 3) \rtimes \pi_{\epsilon}(1,3) - \St(\rho, 3) \rtimes \pi_{-\epsilon}(1,3).
\end{align*}
\end{itemize}
Hence
\[
\semi\Jac_{P_d}(\pi_\epsilon) = \rho|\cdot|^1 \otimes \Pi_{\ub{x}(1,0,0,0)}^{\epsilon}.
\]

\item
When $m = 2$, we have $K_\phi^{(2)} = \{(2,0,0,0) \geq (1,1,0,0)\}$ and 
$(m_{\ub{k}, \ub{k'}}) = \diag(1,2)$.
Moreover: 
\begin{itemize}
\item
If $\ub{k} = (2,0,0,0)$, 
then $\Delta_{\ub{x}(2,0,0,0)} = \pair{\rho; 1,0}$ and 
\begin{align*}
\Pi_{\ub{x}(2,0,0,0)}^{\epsilon} 
= 
&3 \cdot \rho|\cdot|^1 \rtimes \pi_{\epsilon}(3,3) + \rho|\cdot|^1 \rtimes \pi_{-\epsilon}(3,3).
\end{align*}

\item
If $\ub{k} = (1,1,0,0)$, 
then $\Delta_{\ub{x}(1,1,0,0)} = \rho|\cdot|^1 \times \rho|\cdot|^1$ and
\begin{align*}
\Pi_{\ub{x}(1,1,0,0)}^{\epsilon} 
= 2 
&\Big(
2 \cdot \pair{\rho; 1,0} \times \pair{\rho; 1,0} \rtimes \1
\\&
+ 2 \cdot \pair{\rho; 1,0} \rtimes \pi_{\epsilon}(3,3)
- 2 \cdot \pair{\rho; 1,0} \rtimes \pi_{-\epsilon}(3,3)
\\&+ 
2 \cdot \pi_{(\epsilon, \epsilon)}(1,1,3,3)
+ \pi_{(\epsilon, -\epsilon)}(1,1,3,3) + \pi_{(-\epsilon, \epsilon)}(1,1,3,3)
\Big), 
\end{align*}
where $\pi_{(\epsilon_1, \epsilon_3)}(1,1,3,3) 
= \pi((\rho \boxtimes S_1)^{\oplus2} \oplus (\rho \boxtimes S_3)^{\oplus2}, \eta_{(\epsilon_1, \epsilon_3)})$
with $\eta_{(\epsilon_1, \epsilon_3)}(\alpha_{\rho \boxtimes S_i}) = \epsilon_i$.
\end{itemize}
Hence
\begin{align*}
\semi\Jac_{P_{2d}}(\pi_\epsilon) 
&= 
\pair{\rho; 1,0} \otimes \Pi_{\ub{x}(2,0,0,0)}^{\epsilon}
\\&+ 
\half{1}(\rho|\cdot|^1 \times \rho|\cdot|^1) \otimes \Pi_{\ub{x}(1,1,0,0)}^{\epsilon}.
\end{align*}

\item
When $m = 3$, we have $K_\phi^{(3)} = \{(3,0,0,0) \geq (2,1,0,0) \geq (1,1,1,0)\}$ and 
$(m_{\ub{k}, \ub{k'}}) = \diag(1,1,6)$.
Moreover: 
\begin{itemize}
\item
If $\ub{k} = (3,0,0,0)$, 
then $\Delta_{\ub{x}(3,0,0,0)} = \St(\rho,3)$ and 
\begin{align*}
\Pi_{\ub{x}(3,0,0,0)}^{\epsilon} 
= 
&3 \cdot \pi_{\epsilon}(3,3) + \pi_{-\epsilon}(3,3).
\end{align*}

\item
If $\ub{k} = (2,1,0,0)$, 
then $\Delta_{\ub{x}(2,1,0,0)} = \pair{\rho; 1,0} \times \rho|\cdot|^1$ and 
\begin{align*}
\Pi_{\ub{x}(2,1,0,0)}^{\epsilon} 
= 
&4 \cdot \rho|\cdot|^1 \times \pair{\rho; 1,0} \rtimes \1
+ 2 \cdot \rho|\cdot|^1 \rtimes \pi_{\epsilon}(3,1) - 2 \cdot \rho|\cdot|^1 \rtimes \pi_{-\epsilon}(3,1)
\\&+ 3 \cdot \pi_{\epsilon}(3,3) + \pi_{-\epsilon}(3,3).
\end{align*}

\item
If $\ub{k} = (1,1,1,0)$, 
then $\Delta_{\ub{x}(1,1,1,0)} = \rho|\cdot|^1 \times \rho|\cdot|^1 \times \rho|\cdot|^1$ and 
\begin{align*}
\Pi_{\ub{x}(\ub{k})}^{\epsilon} 
= 6 
&\Big(
3 \cdot \pair{\rho; 1,0} \rtimes \pi_{\epsilon}(1,1)
+ \pair{\rho; 1,0} \rtimes \pi_{-\epsilon}(1,1)
\\&+ 
\rho \rtimes \pi_{\epsilon}(1,3) 
- \rho \rtimes \pi_{-\epsilon}(1,3)
\Big).
\end{align*}
\end{itemize}
Hence
\begin{align*}
\semi\Jac_{P_{3d}}(\pi_\epsilon) 
&= 
\St(\rho,3) \otimes \Pi_{\ub{x}(3,0,0,0)}^{\epsilon}
\\&+ (\pair{\rho; 1,0} \times \rho|\cdot|^1) \otimes \Pi_{\ub{x}(2,1,0,0)}^{\epsilon}
\\&+ \frac{1}{6}(\rho|\cdot|^1 \times \rho|\cdot|^1 \times \rho|\cdot|^1) \otimes \Pi_{\ub{x}(1,1,1,0)}^{\epsilon}.
\end{align*}

\item
When $m = 4$, we have $K_\phi^{(4)} = \{(3,1,0,0) \geq (2,2,0,0) \geq (2,1,1,0) \geq (1,1,1,1)\}$ 
and $(m_{\ub{k}, \ub{k'}})= \diag(1,2,2,24)$.
Moreover: 
\begin{itemize}
\item
If $\ub{k} = (3,1,0,0)$, 
then $\Delta_{\ub{x}(3,1,0,0)} = \St(\rho,3) \times \rho|\cdot|^1$ and 
\begin{align*}
\Pi_{\ub{x}(3,1,0,0)}^{\epsilon} 
= 
&4 \cdot \pair{\rho; 1,0} \rtimes \1
+ 2 \cdot \pi_{\epsilon}(3,1) - 2 \cdot \pi_{-\epsilon}(3,1).
\end{align*}

\item
If $\ub{k} = (2,2,0,0)$, then $\Delta_{\ub{x}(2,2,0,0)} = \pair{\rho; 1,0} \times \pair{\rho; 1,0}$ and 
\begin{align*}
\Pi_{\ub{x}(2,2,0,0)}^{\epsilon} 
= 
&4 \cdot \pair{\rho; 1,0} \times \pair{\rho; 1,0} \rtimes \1.
\end{align*}

\item
If $\ub{k} = (2,1,1,0)$, 
then $\tau_{\ub{x}(2,1,1,0)} = \pair{\rho; 1,0} \times \rho|\cdot|^1 \times \rho|\cdot|^1$ and 
\begin{align*}
\Pi_{\ub{x}(2,1,1,0)}^{\epsilon} 
= 2 
&\Big(
4 \cdot \pair{\rho; 1,0} \rtimes \1
+ 3 \cdot \rho|\cdot|^1 \rtimes \pi_{\epsilon}(1,1)
+\rho|\cdot|^1 \rtimes \pi_{-\epsilon}(1,1)
\\&+ 
2 \cdot \pi_{\epsilon}(1,3) 
- 2 \cdot \pi_{-\epsilon}(1,3)
\Big).
\end{align*}

\item
If $\ub{k} = (1,1,1,1)$, 
then $\tau_{\ub{x}(1,1,1,1)} = \rho|\cdot|^1 \times \rho|\cdot|^1 \times \rho|\cdot|^1 \times \rho|\cdot|^1$ and 
\begin{align*}
\Pi_{\ub{x}(\ub{k})}^{\epsilon} 
= 24 \cdot \rho \rtimes \pi_{\epsilon}(1,1). 
\end{align*}
\end{itemize}
Hence
\begin{align*}
\semi\Jac_{P_{4d}}(\pi_\epsilon) 
&= 
(\St(\rho,3) \times \rho|\cdot|^1) \otimes \Pi_{\ub{x}(3,1,0,0)}^{\epsilon}
\\&+ \frac{1}{2}(\pair{\rho; 1,0} \times \pair{\rho; 1,0}) \otimes \Pi_{\ub{x}(2,2,0,0)}^{\epsilon}
\\&+ \frac{1}{2}(\pair{\rho; 1,0} \times \rho|\cdot|^1 \times \rho|\cdot|^1) \otimes \Pi_{\ub{x}(2,1,1,0)}^{\epsilon}
\\&+ \frac{1}{24}(\rho|\cdot|^1 \times \rho|\cdot|^1 \times \rho|\cdot|^1 \times \rho|\cdot|^1) 
\otimes \Pi_{\ub{x}(1,1,1,1)}^{\epsilon}.
\end{align*}

\item
When $m = 5$, we have $K_\phi^{(5)} = \{(3,2,0,0) \geq (3,1,1,0) \geq (2,2,1,0) \geq (2,1,1,1)\}$ 
and $(m_{\ub{k}, \ub{k'}}) = \diag(1,2,2,6)$.
Moreover: 
\begin{itemize}
\item
If $\ub{k} = (3,2,0,0)$, then $\Delta_{\ub{x}(3,2,0,0)} = \St(\rho,3) \times \pair{\rho; 1,0}$ and
\begin{align*}
\Pi_{\ub{x}(3,2,0,0)}^{\epsilon} 
= 
&4 \cdot \rho|\cdot|^1 \rtimes \1_{\SO_1(F)}.
\end{align*}

\item
If $\ub{k} = (3,1,1,0)$, 
then $\Delta_{\ub{x}(3,1,1,0)} = \St(\rho,3) \times \rho|\cdot|^1 \times \rho|\cdot|^1$ and 
\begin{align*}
\Pi_{\ub{x}(3,1,1,0)}^{\epsilon} 
= 2 
&\Big(
3 \cdot \pi_{\epsilon}(1,1) + \pi_{-\epsilon}(1,1)
\Big).
\end{align*}

\item
If $\ub{k} = (2,2,1,0)$, 
then $\Delta_{\ub{x}(2,2,1,0)} = \pair{\rho; 1,0} \times \pair{\rho; 1,0} \times \rho|\cdot|^1$, and 
\begin{align*}
\Pi_{\ub{x}(\ub{k})}^{\epsilon} 
= 
&8 \cdot \rho|\cdot|^1 \rtimes \1_{\SO_1(F)}.
\end{align*}

\item
If $\ub{k} = (2,1,1,1)$, 
then $\tau_{\ub{x}(2,1,1,1)} = \pair{\rho; 1,0} \times \rho|\cdot|^1 \times \rho|\cdot|^1 \times \rho|\cdot|^1$ and 
\begin{align*}
\Pi_{\ub{x}(2,1,1,1)}^{\epsilon} 
= 6 
&\Big(
3 \cdot \pi_{\epsilon}(1,1) + \pi_{-\epsilon}(1,1)
\Big).
\end{align*}
\end{itemize}
Hence
\begin{align*}
\semi\Jac_{P_{5d}}(\pi_\epsilon) 
&= 
(\St(\rho,3) \times \pair{\rho; 1,0}) \otimes \Pi_{\ub{x}(3,2,0,0)}^{\epsilon}
\\&+ \frac{1}{2}(\St(\rho,3) \times \rho|\cdot|^1 \times \rho|\cdot|^1) \otimes \Pi_{\ub{x}(3,1,1,0)}^{\epsilon}
\\&+ \frac{1}{2}(\pair{\rho; 1,0} \times \pair{\rho; 1,0} \times \rho|\cdot|^1) 
\otimes \Pi_{\ub{x}(2,2,1,0)}^{\epsilon}
\\&+ \frac{1}{6}(\pair{\rho; 1,0} \times \rho|\cdot|^1 \times \rho|\cdot|^1 \times \rho|\cdot|^1) 
\otimes \Pi_{\ub{x}(2,1,1,1)}^{\epsilon}.
\end{align*}

\item
When $m = 6$, 
we have $K_\phi^{(6)} = \{(3,3,0,0) \geq (3,2,1,0) \geq (3,1,1,1) \geq (2,2,2,0) \geq (2,2,1,1)\}$ 
and $(m_{\ub{k}, \ub{k'}}) = \diag(2,1,6,6,4)$.
Moreover: 
\begin{itemize}
\item
If $\ub{k} = (3,3,0,0)$, 
then $\Delta_{\ub{x}(3,3,0,0)} = \St(\rho,3) \times \St(\rho,3)$ and
\begin{align*}
\Pi_{\ub{x}(3,3,0,0)}^{\epsilon} 
= 
&4 \cdot \1_{\SO_1(F)}.
\end{align*}

\item
If $\ub{k} = (3,2,1,0)$, 
then $\Delta_{\ub{x}(3,2,1,0)} = \St(\rho,3) \times \pair{\rho; 1,0} \times \rho|\cdot|^1$ and
\begin{align*}
\Pi_{\ub{x}(3,2,1,0)}^{\epsilon} 
= 
&4 \cdot \1_{\SO_1(F)}.
\end{align*}

\item
If $\ub{k} = (3,1,1,1)$, then $\Pi_{\ub{x}(3,1,1,1)}^{\epsilon} = 0$.

\item
If $\ub{k} = (2,2,2,0)$, then $\Pi_{\ub{x}(2,2,2,0)}^{\epsilon} = 0$.

\item
If $\ub{k} = (2,2,1,1)$, 
then $\Delta_{\ub{x}(2,2,1,1)} = \pair{\rho; 1,0} \times \pair{\rho; 1,0} \times \rho|\cdot|^1 \times \rho|\cdot|^1$, 
and 
\begin{align*}
\Pi_{\ub{x}(2,2,1,1)}^{\epsilon} 
= 
&8 \cdot \1_{\SO_1(F)}.
\end{align*}
\end{itemize}
Hence
\begin{align*}
\semi\Jac_{P_{6d}}(\pi_\epsilon) 
&= 
2(\St(\rho,3) \times \St(\rho, 3)) \otimes \1_{\SO_1(F)}
\\&+ 4(\St(\rho,3) \times \pair{\rho; 1,0} \times \rho|\cdot|^1) \otimes \1_{\SO_1(F)}
\\&+ 2(\pair{\rho; 1,0} \times \pair{\rho; 1,0} \times \rho|\cdot|^1 \times \rho|\cdot|^1) \otimes \1_{\SO_1(F)}.
\end{align*}
\end{enumerate}
\end{ex}
\fi

\begin{rem}
In Theorem \ref{jac2}, 
one can replace $\Delta_{\ub{x}(\ub{k})}$ with its unique irreducible subrepresentation $\tau_{\ub{x}(\ub{k})}$.
Then one should consider the matrix 
$(M_{\ub{k}, \ub{l}}) = (\dim_\C \Jac_{\rho|\cdot|^{\ub{x}(\ub{k})}}(\tau_{\ub{x}(\ub{l})}))$.
One might seem that $(M_{\ub{k}, \ub{l}})$ easier than $(m_{\ub{k}, \ub{l}})$.
For instance, if $\phi$ is in Example \ref{ab}, 
all $(M_{\ub{k}, \ub{l}})$ are the identity matrix, but not so is some $(m_{\ub{k}, \ub{l}})$.
However, in general, $(M_{\ub{k}, \ub{l}})$ is not always diagonal.
In Example \ref{244}, such a non-diagonal $(M_{\ub{k}, \ub{l}})$ appears.
\end{rem}

\section{Proof of main theorems}\label{pf}
In this section, we prove main theorems (Theorems \ref{jac1} and \ref{jac2}).

\subsection{The case of higher multiplicity}
We prove Theorem \ref{jac1} (1) in this subsection.
It immediately follows from the case of $\phi \in \Phi_\disc(G)$ 
together with Tadi{\'c}'s formula (Corollary \ref{cor.tdc}).

\begin{proof}[Proof of Theorem \ref{jac1} (1)]
We prove the assertion by induction on $m$.
Since 
\[
\pi(\phi, \eta) = \St(\rho, 2x+1) \rtimes \pi(\phi - (\rho \boxtimes S_{2x+1})^{\oplus 2}, \eta),
\] 
by Corollary \ref{cor.tdc}, we have
\begin{align*}
\Jac_{\rho|\cdot|^x} (\pi(\phi, \eta)) 
&= 
2 \cdot \pair{\rho; x,\dots,-(x-1)} \rtimes \pi \left(\phi - (\rho \boxtimes S_{2x+1})^{\oplus 2}, \eta \right)
\\&+
\St(\rho, 2x+1) \rtimes \Jac_{\rho|\cdot|^{x}} \left(\pi \left(\phi - (\rho \boxtimes S_{2x+1})^{\oplus 2}, \eta \right)\right).
\end{align*}
This proves the assertion when $m=3$ or $m=4$.
When $m \geq 5$, 
since 
\begin{align*}
&\St(\rho, 2x+1) \times \pair{\rho; x,\dots,-(x-1)} \rtimes \pi \left(\phi - (\rho \boxtimes S_{2x+1})^{\oplus 4}, \eta \right)
\\&\cong
\pair{\rho; x,\dots,-(x-1)} \times \St(\rho, 2x+1) \rtimes \pi 
\left(\phi_0 \oplus (\phi - (\rho \boxtimes S_{2x+1})^{\oplus 4}), \eta \right)
\\&\cong
\pair{\rho; x,\dots,-(x-1)} \rtimes \pi \left(\phi - (\rho \boxtimes S_{2x+1})^{\oplus 2}, \eta \right), 
\end{align*}
we obtain the assertion by the induction hypothesis.
\end{proof}

\subsection{The case of multiplicity one}
Next, we prove Theorem \ref{jac1} (2).
Let $\phi = \phi_0 \oplus (\rho \boxtimes S_{2x+1})$ with $\rho \boxtimes S_{2x+1} \not\subset \phi_0$, 
and $\eta \in \widehat{\AA_\phi}$.
Set 
\[
\phi' = \phi - (\rho \boxtimes S_{2x+1}) \oplus (\rho \boxtimes S_{2x-1}).
\]

\begin{proof}[Proof of Theorem \ref{jac1} (2)]
First, we assume that $x > 0$ and $\pi(\phi', \eta') \not= 0$.
We apply M{\oe}glin's construction to $\Pi_{\phi'}$.
Write 
\[
\phi' = \left(\bigoplus_{i=1}^{t} \rho \boxtimes S_{a_i} \right) \oplus \phi'_e
\]
with $a_1 \leq \dots \leq a_t$ and $\rho \boxtimes S_a \not\subset \phi'_e$ for any $a>0$.
Set 
\[
t_0 = \max\{i \in \{1, \dots, t\}\ |\ t_i = 2x-1\}.
\]
Take a new $L$-parameter
\[
\phi'_\gg = \left(\bigoplus_{i=1}^{t} \rho \boxtimes S_{a'_i} \right) \oplus \phi'_e
\]
such that
\begin{itemize}
\item
$a'_1 < \dots < a'_t$; 
\item
$a_i' \geq a_i$ and $a'_i \equiv a_i \bmod 2$ for any $i$; 
\item
$a'_{t_0} \geq 2x+1$.
\end{itemize}
We can identify $\AA_{\phi'_\gg}$ with $\AA_{\phi'}$ canonically.
Let $\eta'_\gg \in \widehat{\AA_{\phi'_\gg}}$ be the character corresponding to $\eta' \in \widehat{\AA_{\phi'}}$.
Then Theorem \ref{moe} says that 
\[
\pi(\phi', \eta') = 
\Jac_{\rho|\cdot|^{\half{a_t'-1}}, \dots, \rho|\cdot|^\half{a_t+1}}
\circ \dots \circ
\Jac_{\rho|\cdot|^{\half{a_1'-1}}, \dots, \rho|\cdot|^\half{a_1+1}}
(\pi(\phi'_\gg, \eta'_\gg)).
\]
\par

When $i=t_0$, we note that
\[
\Jac_{\rho|\cdot|^{\half{a_{t_0}'-1}}, \dots, \rho|\cdot|^\half{a_{t_0}+1}}
=
\Jac_{\rho|\cdot|^{x}} \circ \Jac_{\rho|\cdot|^{\half{a_{t_0}'-1}}, \dots, \rho|\cdot|^\half{a_{t_0}+3}}.
\]
Since $\phi'$ does not contain $\rho \boxtimes S_{2x+1}$, 
for $i>t_0$ and $a_i < 2x'+1 \leq a'_i$ with $2x'+1 \equiv a_i \bmod 2$, 
we have $x'-x > 1$.
By Lemma \ref{lem5} (2), 
we see that $\pi(\phi', \eta')$ is the image of 
\[
\Jac_{\rho|\cdot|^{\half{a_t'-1}}, \dots, \rho|\cdot|^\half{a_t+1}}
\circ \dots \circ 
\Jac_{\rho|\cdot|^{\half{a_{t_0}'-1}}, \dots, \rho|\cdot|^\half{a_{t_0}+3}}
\circ \dots \circ 
\Jac_{\rho|\cdot|^{\half{a_1'-1}}, \dots, \rho|\cdot|^\half{a_1+1}}
(\pi(\phi'_\gg, \eta'_\gg))
\]
under $\Jac_{\rho|\cdot|^x}$.
However, by applying Theorem \ref{moe} again, 
we see that this representation is isomorphic to $\pi(\phi, \eta)$.
Therefore $\pi(\phi', \eta') = \Jac_{\rho|\cdot|^x} (\pi(\phi, \eta))$, as desired.
\par

Next, we assume that $\pi(\phi', \eta') = 0$.
We claim that $\Jac_{\rho|\cdot|^x}(\pi(\phi, \eta)) = 0$.
When $\phi \in \Phi_\disc(G)$, this was proven in \cite[Lemma 7.3]{X1}.
When $\phi \in \Phi_\gp(G) \setminus \Phi_\disc(G)$, 
there exists an irreducible representation $\phi_1$ which is contained in $\phi$ with multiplicity at least two.
Then $\pi(\phi, \eta)$ is a subrepresentation of $\tau_1 \rtimes \pi(\phi-\phi_1^{\oplus2}, \eta)$, 
where $\tau_1$ is the irreducible tempered representation of $\GL_k(F)$
corresponding to $\phi_1$.
Since $\rho \boxtimes S_{2x+1}$ is contained in $\phi$ with multiplicity one, 
we have $\phi_1 \not\cong \rho \boxtimes S_{2x+1}$.
This implies that 
\[
\Jac_{\rho|\cdot|^x}(\tau_1 \rtimes \pi(\phi-\phi_1^{\oplus2}, \eta)) 
= 
\tau_1 \rtimes \Jac_{\rho|\cdot|^x}(\pi(\phi-\phi_1^{\oplus2}, \eta)). 
\]
By the induction hypothesis, $\Jac_{\rho|\cdot|^x}(\pi(\phi-\phi_1^{\oplus2}, \eta)) = 0$ 
unless $\phi_1 = \rho \boxtimes S_{2x-1}$ and $\phi$ contains it with multiplicity exactly two.
In this case, one can take $\eta_- \in \widehat{\AA_\phi}$ such that 
$\pi(\phi, \eta_-) \not= 0$ and 
\[
\pi(\phi, \eta) \oplus \pi(\phi, \eta_-) = \St(\rho, 2x-1) \rtimes \pi(\phi-\phi_1^{\oplus2}, \eta).
\]
Then by the first case, we see that $\Jac_{\rho|\cdot|^x}(\pi(\phi, \eta_-)) \not= 0$ and 
$\St(\rho, 2x-1) \rtimes \Jac_{\rho|\cdot|^x}(\pi(\phi-\phi_1^{\oplus2}, \eta))$ is irreducible.
Hence $\Jac_{\rho|\cdot|^x}(\pi(\phi, \eta))$ must be zero.
This completes the proof of Theorem \ref{jac1} (2).
\end{proof}
By the same argument as the last part, 
one can prove that $\Jac_{\rho}(\pi(\phi, \eta)) = 0$ when $x = 0$ and $m = 1$.

\subsection{Description of small standard modules}
Before proving Theorem \ref{jac1} (3), 
we describe the structures of some standard modules.

\begin{lem}\label{l2}
Let $\phi \in \Phi_\disc(G)$.
Suppose that $x > 0$, and $\phi \supset \rho \boxtimes S_{2x-1}$ but $\phi \not\supset \rho \boxtimes S_{2x+1}$.
Let $\eta \in \widehat{\AA_\phi}$ such that $\pi(\phi, \eta) \not= 0$.
We set 
\begin{itemize}
\item
$\Pi = \rho|\cdot|^x \rtimes \pi(\phi, \eta)$ to be a standard module; 
\item
$\sigma$ to be the unique irreducible quotient of $\Pi$; 
\item
$\phi' = \phi - (\rho \boxtimes S_{2x-1}) \oplus (\rho \boxtimes S_{2x+1})$; 
\item
$\eta' \in \widehat{\AA_{\phi'}}$ to be the character corresponding to $\eta \in \widehat{\AA_\phi}$
via the canonical identification $\AA_{\phi'} = \AA_\phi$.
\end{itemize}
Then there exists an exact sequence
\[
\begin{CD}
0 @>>> \pi(\phi', \eta') @>>> \Pi @>>> \sigma @>>> 0.
\end{CD}
\]
In particular, $\Pi$ has length two.
\end{lem}
\begin{proof}
We note that $\pi(\phi', \eta')$ is an irreducible subrepresentation of $\Pi$
by Theorem \ref{jac1} (2) and Lemma \ref{lem5} (1).
\par

If $\sigma'$ is an irreducible subquotient of $\Pi$ which is non-tempered, 
by Tadi{\'c}'s formula and Casselman's criterion, there exists a maximal parabolic subgroup $P_k$ of $G(F)$ such that 
$\semi\Jac_{P_k}(\sigma')$ contains an irreducible representation of the form $(\rho|\cdot|^{-x} \times \tau) \boxtimes \sigma_0$. 
In particular, we have $\Jac_{\rho|\cdot|^{-x}}(\sigma') \not= 0$.
However, since $\Jac_{\rho|\cdot|^{-x}}(\Pi) = \pi(\phi, \eta)$ is irreducible, 
we see that $\sigma' = \sigma$, i.e., $\Pi$ has only one irreducible non-tempered subquotient.
\par

Let $\Pi^{\sub}$ be the maximal proper subrepresentation of $\Pi$, i.e., $\Pi/\Pi^{\sub} \cong \sigma$.
By the above argument, all irreducible subquotients of $\Pi^\sub$ must be tempered.
Since they have the same cuspidal support, they share the same Plancherel measure.
This implies that 
all irreducible subquotients of $\Pi^\sub$
belong to the same $L$-packet $\Pi_{\phi'}$ (see \cite[Lemma A.6]{GI2}), 
so that they are all discrete series.
Hence $\Pi^\sub$ is semisimple.
In particular, any irreducible subquotient $\pi'$ of $\Pi^\sub$ is a subrepresentation of $\Pi$, 
so that $\Jac_{\rho|\cdot|^x}(\pi') \not= 0$.
However, since $\Jac_{\rho|\cdot|^{x}}(\Pi) = \pi(\phi, \eta)$ is irreducible, 
$\Pi$ has only one irreducible subrepresentation.
Therefore $\Pi^\sub = \pi(\phi', \eta')$.
This completes the proof.
\end{proof}

We describe the standard module appearing in Theorem \ref{jac1} (3).
When $x=1/2$, the standard module was described in Lemma \ref{l2}.
Hence we assume $x > 1/2$.

\begin{prop}\label{sci}
Let $\phi \in \Phi_\gp(G)$.
Suppose that $x > 1/2$ and $\phi \not\supset \rho \boxtimes S_{2x+1}$.
Let $\eta \in \widehat{\AA_\phi}$ such that $\pi(\phi, \eta) \not= 0$.
We set 
\begin{itemize}
\item
$\Pi = \pair{\rho; x, x-1, \dots, -(x-1)} \rtimes \pi(\phi, \eta)$ to be a standard module; 
\item
$\sigma$ to be the unique irreducible quotient of $\Pi$; 
\item
$\phi' = \phi \oplus (\rho \boxtimes S_{2x-1}) \oplus (\rho \boxtimes S_{2x+1})$; 
\item
$\eta'_+$ and $\eta'_-$ to be the two distinct characters of $\widehat{\AA_{\phi'}}$ 
such that $\eta'_\pm|\AA_{\phi} = \eta$ and $\eta'_\pm(z_{\phi'}) = 1$.
\end{itemize}
Then there exists an exact sequence
\[
\begin{CD}
0 @>>> \pi(\phi', \eta'_+) \oplus \pi(\phi', \eta'_-) @>>> \Pi @>>> \sigma @>>> 0.
\end{CD}
\]
In particular, $\Pi$ has length $2$ or $3$ according to $\phi \supset \rho \boxtimes S_{2x-1}$ or not. 
\end{prop}
\begin{proof}
First, we show that there is an inclusion $\pi(\phi', \eta'_\epsilon) \hookrightarrow \Pi$ for each $\epsilon \in \{\pm1\}$.
To do this, we may assume that $\pi(\phi', \eta'_\epsilon) \not= 0$.
Note that $\phi'$ contains $\rho \boxtimes S_{2x+1}$ with multiplicity one.
By Theorem \ref{jac1} (2), 
we see that $\Jac_{\rho|\cdot|^x}(\pi(\phi', \eta'_\epsilon))$ 
is nonzero and is an irreducible subrepresentation of $\St(\rho, 2x-1) \rtimes \pi(\phi, \eta)$.
By Lemma \ref{lem5} (1), we have an inclusion
\[
\pi(\phi', \eta'_\epsilon) \hookrightarrow \rho|\cdot|^x \times \St(\rho, 2x-1) \rtimes \pi(\phi, \eta).
\]
Since $\Pi$ is a subrepresentation of $\rho|\cdot|^x \times \St(\rho, 2x-1) \rtimes \pi(\phi, \eta)$
such that
\[
\Jac_{\rho|\cdot|^x}\left(
\rho|\cdot|^x \times \St(\rho, 2x-1) \rtimes \pi(\phi, \eta)
\right) = \Jac_{\rho|\cdot|^x}(\Pi), 
\]
the above inclusion factors through $\pi(\phi', \eta'_\epsilon) \hookrightarrow \Pi$.
\par

If $\semi\Jac_{P_k}(\Pi)$ contains an irreducible representation $\tau \boxtimes \pi_0$ such that the central character of $\tau$ 
is of the form $\chi|\cdot|^{s}$ with $\chi$ unitary and $s<0$, 
by Tadi{\'c}'s formula (Theorem \ref{tdc}) and Casselman's criterion, 
$\tau = |\cdot|^{-\half{1}}\St(\rho, 2x) \times (\times_{i =1}^r \tau_i)$,
where $\tau_i$ is a discrete series representation of $\GL_{k_i}(F)$ 
such that the corresponding irreducible representation $\phi_i$ of $W_F \times \SL_2(\C)$ 
is contained in $\phi$ with multiplicity at least two, 
and $\pi_0 = \pi(\phi_0, \eta_0)$ 
with $\phi_0 = \phi- (\oplus_{i=1}^r \phi_i)^{\oplus 2}$ and $\eta_0 = \eta|\AA_{\phi_0}$.
Since such an irreducible representation $\tau \boxtimes \pi_0$ is also contained in $\semi\Jac_{P_k}(\sigma)$, 
we see that $\sigma$ is the unique irreducible non-tempered subquotient of $\Pi$.
Namely, if we let $\Pi^\sub$ be the maximal proper subrepresentation of $\Pi$, i.e., $\Pi/\Pi^\sub \cong \sigma$, 
then all irreducible subquotients of $\Pi^\sub$ must be tempered.
Moreover since these irreducible subquotients have the same cuspidal support
so that they share the same Plancherel measure, 
they belong to the same $L$-packet $\Pi_{\phi'}$ (see \cite[Lemma A.6]{GI2}).
\par

Now we show that $\Pi^\sub$ is isomorphic to $\pi(\phi', \eta'_+) \oplus \pi(\phi', \eta'_-)$.
We separate the cases as follows:
\begin{itemize}
\item
$\phi$ is discrete and $\rho \boxtimes S_{2x-1} \not\subset \phi$; 
\item
$\phi$ is discrete and $\rho \boxtimes S_{2x-1} \subset \phi$; 
\item
$\phi$ is general.
\end{itemize}
\par

When $\phi$ is discrete and $\rho \boxtimes S_{2x-1} \not\subset \phi$, 
we note that $\phi' = \phi \oplus (\rho \boxtimes S_{2x-1}) \oplus (\rho \boxtimes S_{2x+1})$ is also discrete.
Then since all irreducible subquotients of $\Pi^\sub$ are discrete series, $\Pi^\sub$ is semisimple.
In particular, any irreducible subquotient $\pi'$ of $\Pi^\sub$ 
is a subrepresentation of $\Pi$ so that $\Jac_{\rho|\cdot|^x}(\pi') \not= 0$.
However, since $\Jac_{\rho|\cdot|^x}(\Pi) = \Jac_{\rho|\cdot|^x}(\pi(\phi', \eta'_+) \oplus \pi(\phi', \eta'_-))$, 
we have $\Pi^\sub \cong \pi(\phi', \eta'_+) \oplus \pi(\phi', \eta'_-)$.
\par

When $\phi$ is discrete and $\rho \boxtimes S_{2x-1} \subset \phi$, 
any irreducible subquotient $\pi'$ of $\Pi^\sub$ belongs to $\Pi_{\phi'}$ 
with $\phi' = \phi_0' \oplus (\rho \boxtimes S_{2x-1})^{\oplus 2}$, 
where $\phi_0' = \phi - (\rho \boxtimes S_{2x-1}) \oplus (\rho \boxtimes S_{2x+1})$ is discrete 
such that $\rho \boxtimes S_{2x-1} \not\subset \phi_0'$.
Hence the Jacquet module $\semi\Jac_{P_{d(2x-1)}}(\pi')$ contains an irreducible representation of the form $\St(\rho, 2x-1) \otimes \pi_0'$.
By Tadi{\'c}'s formula (Corollary \ref{cor.tdc}), 
the sum of irreducible representations of this form appearing in $\semi\Jac_{P_{d(2x-1)}}(\Pi)$ is 
\[
\St(\rho, 2x-1) \otimes \semi( \rho|\cdot|^x \rtimes \pi(\phi, \eta) ).
\]
By Lemma \ref{l2}, we have an exact sequence
\[
\begin{CD}
0 @>>> \pi(\phi_0', \eta_0') @>>> \rho|\cdot|^x \rtimes \pi(\phi, \eta) @>>> \sigma' @>>> 0, 
\end{CD}
\]
where $\sigma'$ is the unique irreducible quotient of $\rho|\cdot|^x \rtimes \pi(\phi, \eta)$, 
and $\eta_0' \in \widehat{\AA_{\phi_0'}}$ is the character corresponding to $\eta \in \widehat{\AA_\phi}$.
Now there exists $\epsilon \in \{\pm1\}$ such that $\pi(\phi', \eta'_{-\epsilon}) = 0$.
Moreover, 
$\semi\Jac_{P_{d(2x-1)}}(\pi(\phi', \eta'_\epsilon)) \supset \St(\rho, 2x-1) \otimes \pi(\phi_0', \eta_0')$ 
since
\[
\eta'_\epsilon(\alpha_{\rho \boxtimes S_{2x+1}})
= \eta'_\epsilon(\alpha_{\rho \boxtimes S_{2x-1}})
= \eta(\alpha_{\rho \boxtimes S_{2x-1}})
= \eta_0(\alpha_{\rho \boxtimes S_{2x-1}}).
\]
On the other hand, since $\sigma \hookrightarrow \St(\rho, 2x-1) \times \rho|\cdot|^{-x} \rtimes \pi(\phi, \eta)$, 
we see that $\Jac_{P_{d(2x-1)}}(\sigma)$ is nonzero and contains $\St(\rho, 2x-1) \otimes \sigma'$.
Hence 
\[
\semi\Jac_{P_{d(2x-1)}}(\Pi) - \semi\Jac_{P_{d(2x-1)}}(\pi(\phi', \eta'_\epsilon)) - \semi\Jac_{P_{d(2x-1)}}(\sigma) 
\]
has no irreducible representation of the form $\St(\rho, 2x-1) \otimes \pi_0'$.
This shows that $\Pi^\sub = \pi(\phi', \eta'_\epsilon)$.
\par

In general, we prove the claim by induction on the dimension of $\phi$.
When $\phi$ is not discrete, 
there exists an irreducible representation $\phi_1$ of $W_F \times \SL_2(\C)$ which $\phi$ contains with multiplicity at least two.
Note that $\phi_1 \not\cong \rho \boxtimes S_{2x+1}$.
Set $\phi_0 = \phi - \phi_1^{\oplus 2}$, and $\eta_0 = \eta|\AA_{\phi_0}$.
Take $\Pi_0$, $\sigma_0$, $\phi_0'$ and $\eta'_{0,\epsilon} \in \widehat{\AA_{\phi_0'}}$ as in the statement of the proposition.
By induction hypothesis, we have an exact sequence
\[
\begin{CD}
0 @>>> \pi(\phi_0', \eta'_{0,+}) \oplus \pi(\phi_0', \eta'_{0,-}) @>>> \Pi_0 @>>> \sigma_0 @>>> 0.
\end{CD}
\]
Let $\tau$ be the irreducible discrete series representation of $\GL_k(F)$ corresponding to $\phi_1$.
The above exact sequence remains exact after taking the parabolic induction functor $\pi_0 \mapsto \tau \rtimes \pi_0$. 
Note that $\tau \times \pair{\rho; x, x-1, \dots, -(x-1)} \cong \pair{\rho; x, x-1, \dots, -(x-1)} \times \tau$ by Theorem \ref{zel}.
Since $\sigma_0$ is unitary, the parabolic induction $\tau \rtimes \sigma_0$ is semisimple.
In particular, any irreducible subquotient of $\tau \rtimes \sigma_0$ is non-tempered.
Considering the cases where
\begin{itemize}
\item
$\phi$ contains $\phi_1$ with multiplicity more than two; 
\item
$\phi$ contains $\phi_1$ with multiplicity exactly two and $\phi_1 \not\cong \rho\boxtimes S_{2x-1}$; 
\item
$\phi$ contains $\phi_1$ with multiplicity exactly two and $\phi_1 \cong \rho\boxtimes S_{2x-1}$ 
\end{itemize}
separately, 
we see that $\Pi^\sub \cong \pi(\phi', \eta'_+) \oplus \pi(\phi', \eta'_-)$ in all cases.
This completes the proof.
\end{proof}

\subsection{The case of multiplicity two}
Finally, we prove Theorem \ref{jac1} (3).

\begin{lem}\label{lem2}
Let $\phi \in \Phi_\gp(G)$, $\eta \in \widehat{\AA_\phi}$ and $x > 0$.
Suppose that 
$\phi$ contains both $\rho \boxtimes S_{2x+1}$ and $\rho \boxtimes S_{2x+3}$ with multiplicity one.
Then we have
\[
\Jac_{\rho|\cdot|^{x+1}, \rho|\cdot|^{x}, \rho|\cdot|^{x}}(\pi(\phi, \eta)) 
\subset
2 \cdot \Jac_{\rho|\cdot|^{x}, \rho|\cdot|^{x+1}, \rho|\cdot|^{x}}(\pi(\phi, \eta)). 
\]
\end{lem}
\begin{proof}
We may assume that $\Jac_{\rho|\cdot|^{x+1}, \rho|\cdot|^{x}, \rho|\cdot|^{x}}(\pi(\phi, \eta)) \not= 0$.
By Lemma \ref{lem5} (1), there exists an irreducible subquotient $\sigma$ of this Jacquet module such that
\[
\pi(\phi, \eta) \hookrightarrow \rho|\cdot|^{x+1} \times \rho|\cdot|^{x} \times \rho|\cdot|^{x} \rtimes \sigma.
\]
Since there exists an exact sequence
\[
\begin{CD}
0 @>>> \pair{\rho; x+1,x} @>>> \rho|\cdot|^{x+1} \times \rho|\cdot|^{x} @>>> \pair{\rho; x,x+1} @>>> 0, 
\end{CD}
\]
where $\pair{\rho; x,x+1}$ is the unique irreducible subrepresentation 
of $\rho|\cdot|^{x} \times \rho|\cdot|^{x+1}$, 
we see that $\pi(\phi, \eta)$ is a subrepresentation of 
$\pair{\rho; x+1,x} \times \rho|\cdot|^{x} \rtimes \sigma$ or $\pair{\rho; x,x+1} \times \rho|\cdot|^{x} \rtimes \sigma$.
Since $\pair{\rho; x+1,x} \times \rho|\cdot|^{x} \cong \rho|\cdot|^x \times \pair{\rho; x+1,x}$ and 
$\pair{\rho; x,x+1} \times \rho|\cdot|^{x} \cong \rho|\cdot|^x \times \pair{\rho; x,x+1}$, 
we have $\Jac_{\rho|\cdot|^x}(\pi(\phi, \eta)) \not= 0$.
\par

By Theorem \ref{jac1} (2), 
$\Jac_{\rho|\cdot|^{x+1}}(\pi(\phi, \eta)) \not= 0$ and $\Jac_{\rho|\cdot|^x}(\pi(\phi, \eta)) \not= 0$ imply that 
$\sigma' = \Jac_{\rho|\cdot|^{x}, \rho|\cdot|^{x+1}, \rho|\cdot|^{x}}(\pi(\phi, \eta))$ is nonzero and irreducible.
Moreover, we have $\Jac_{\rho|\cdot|^{x}}(\sigma') = 0$ and $\Jac_{\rho|\cdot|^{x+1}}(\sigma') = 0$.
By Lemma \ref{lem5} (1), we have an inclusion
\[
\pi(\phi, \eta) \hookrightarrow \rho|\cdot|^{x} \times \rho|\cdot|^{x+1} \times \rho|\cdot|^{x} \rtimes \sigma'.
\]
Since 
\[
\Jac_{\rho|\cdot|^{x+1}, \rho|\cdot|^{x}, \rho|\cdot|^{x}}(\rho|\cdot|^{x} \times \rho|\cdot|^{x+1} \times \rho|\cdot|^{x} \rtimes \sigma')
= 2 \cdot \sigma', 
\]
we have $\Jac_{\rho|\cdot|^{x+1}, \rho|\cdot|^{x}, \rho|\cdot|^{x}}(\pi(\phi, \eta)) \subset 2 \cdot \sigma'$, as desired.
\end{proof}

Suppose that $x > 0$.
Let $\phi = \phi_0 \oplus (\rho \boxtimes S_{2x+1})^{\oplus 2}$ 
with $\rho \boxtimes S_{2x+1} \not\subset \phi_0$, 
and $\eta \in \widehat{\AA_\phi}$.

\begin{lem}\label{2times}
Set $\phi_1 = \phi - (\rho \boxtimes S_{2x+1})^{\oplus 2} \oplus (\rho \boxtimes S_{2x-1})^{\oplus 2}$.
We canonically identify $\AA_{\phi_1}$ with $\AA_\phi$, 
and let $\eta_1 \in \widehat{\AA_{\phi_1}}$ be the character corresponding to $\eta \in \widehat{\AA_\phi}$.
Then we have
\[
\Jac_{\rho|\cdot|^x, \rho|\cdot|^x} (\pi(\phi, \eta)) = 2 \cdot \pi(\phi_1, \eta_1).
\]
\end{lem}
\begin{proof}
Let $\eta_+ = \eta$ and $\eta_- \in \widehat{\AA_\phi}$ be as in the statement of Theorem \ref{jac1} (3).
We also take $\eta_{1,\pm} \in \widehat{\AA_{\phi_1}}$ corresponding to $\eta_{\pm}$.
Then we have
\begin{itemize}
\item
$\pi(\phi, \eta_+) \oplus \pi(\phi, \eta_-) = \St(\rho, 2x+1) \rtimes \pi(\phi_0, \eta|\AA_{\phi_0})$; 
\item
$\pi(\phi_1, \eta_{1,+}) \oplus \pi(\phi_1, \eta_{1,-}) = \St(\rho, 2x-1) \rtimes \pi(\phi_0, \eta|\AA_{\phi_0})$; 
\item
$\Jac_{\rho|\cdot|^x, \rho|\cdot|^x} (\St(\rho, 2x+1) \rtimes \pi(\phi_0, \eta|\AA_{\phi_0}))
\cong 2 \cdot \St(\rho, 2x-1) \rtimes \pi(\phi_0, \eta|\AA_{\phi_0})$.
\end{itemize}
Therefore, it is enough to show that $\Jac_{\rho|\cdot|^x, \rho|\cdot|^x} (\pi(\phi, \eta)) \subset 2 \cdot \pi(\phi_1, \eta_1)$.
\par

We apply M{\oe}glin's construction to $\Pi_{\phi}$.
Write 
\[
\phi = \left(\bigoplus_{i=1}^{t} \rho \boxtimes S_{a_i} \right) \oplus \phi_e
\]
with $a_1 \leq \dots \leq a_t$ and $\rho \boxtimes S_a \not\subset \phi_e$ for any $a>0$.
There exists $t_0 > 1$ such that $a_{t_0-1} = a_{t_0} = 2x+1$.
Take a new $L$-parameter
\[
\phi_\gg = \left(\bigoplus_{i=1}^{t} \rho \boxtimes S_{a'_i} \right) \oplus \phi_e
\]
such that
\begin{itemize}
\item
$a'_1 < \dots < a'_t$; 
\item
$a_i' \geq a_i$ and $a'_i \equiv a_i \bmod 2$ for any $i$. 
\end{itemize}
In particular, $a'_{t_0} \geq 2x+3$.
We can identify $\AA_{\phi_\gg}$ with $\AA_{\phi}$ canonically.
Let $\eta_\gg \in \widehat{\AA_{\phi_\gg}}$ be the character corresponding to $\eta \in \widehat{\AA_{\phi}}$.
Then Theorem \ref{moe} says that 
\[
\pi(\phi, \eta) = 
\Jac_{\rho|\cdot|^{\half{a'_t-1}}, \dots, \rho|\cdot|^\half{a_t+1}}
\circ \dots \circ
\Jac_{\rho|\cdot|^{\half{a'_1-1}}, \dots, \rho|\cdot|^\half{a_1+1}}
(\pi(\phi_\gg, \eta_\gg)).
\]
\par

Note that $(a_{t_0}+1)/2 = x+1$.
By Lemma \ref{lem5} (2), we see that
\begin{align*}
&\Jac_{\rho|\cdot|^x, \rho|\cdot|^x} (\pi(\phi, \eta))
\\&= 
\Jac_{\rho|\cdot|^{\half{a'_t-1}}, \dots, \rho|\cdot|^\half{a_t+1}}
\circ \dots \circ
\Jac_{\rho|\cdot|^{\half{a'_{t_0}-1}}, \dots, \rho|\cdot|^\half{a_{t_0}+3}}
\\&\circ
\Jac_{\rho|\cdot|^{x+1}, \rho|\cdot|^{x}, \rho|\cdot|^{x}}
\left(
\Jac_{\rho|\cdot|^{\half{a'_{t_0-1}-1}}, \dots, \rho|\cdot|^\half{a_{t_0-1}+1}}
\circ \dots \circ
\Jac_{\rho|\cdot|^{\half{a'_1-1}}, \dots, \rho|\cdot|^\half{a_1+1}}
(\pi(\phi_\gg, \eta_\gg))
\right).
\end{align*}
By Lemma \ref{lem2}, 
we have
\begin{align*}
&\Jac_{\rho|\cdot|^{x+1}, \rho|\cdot|^{x}, \rho|\cdot|^{x}}
\left(
\Jac_{\rho|\cdot|^{\half{a'_{t_0-1}-1}}, \dots, \rho|\cdot|^\half{a_{t_0-1}+1}}
\circ \dots \circ
\Jac_{\rho|\cdot|^{\half{a'_1-1}}, \dots, \rho|\cdot|^\half{a_1+1}}
(\pi(\phi_\gg, \eta_\gg))
\right)
\\&\subset 2 \cdot 
\Jac_{\rho|\cdot|^{x}, \rho|\cdot|^{x+1}, \rho|\cdot|^{x}}
\left(
\Jac_{\rho|\cdot|^{\half{a'_{t_0-1}-1}}, \dots, \rho|\cdot|^\half{a_{t_0-1}+1}}
\circ \dots \circ
\Jac_{\rho|\cdot|^{\half{a'_1-1}}, \dots, \rho|\cdot|^\half{a_1+1}}
(\pi(\phi_\gg, \eta_\gg))
\right).
\end{align*}
Since
\begin{align*}
&\Jac_{\rho|\cdot|^{\half{a'_t-1}}, \dots, \rho|\cdot|^\half{a_t+1}}
\circ \dots \circ
\Jac_{\rho|\cdot|^{\half{a'_{t_0}-1}}, \dots, \rho|\cdot|^\half{a_{t_0}+3}}
\\&\circ
\Jac_{\rho|\cdot|^{x}, \rho|\cdot|^{x+1}, \rho|\cdot|^{x}}
\left(
\Jac_{\rho|\cdot|^{\half{a'_{t_0-1}-1}}, \dots, \rho|\cdot|^\half{a_{t_0-1}+1}}
\circ \dots \circ
\Jac_{\rho|\cdot|^{\half{a'_1-1}}, \dots, \rho|\cdot|^\half{a_1+1}}
(\pi(\phi_\gg, \eta_\gg))
\right)
\\&= \pi(\phi_1, \eta_1)
\end{align*}
by Theorem \ref{moe}, 
we have $\Jac_{\rho|\cdot|^x, \rho|\cdot|^x} (\pi(\phi, \eta)) \subset 2 \cdot  \pi(\phi_1, \eta_1)$, as desired.
\end{proof}

Now we can prove Theorem \ref{jac1} (3).
\begin{proof}[Proof of Theorem \ref{jac1} (3)]
By Corollary \ref{cor.tdc}, we have
\[
\Jac_{\rho|\cdot|^x}( \pi(\phi, \eta_+) \oplus \pi(\phi, \eta_-) )
=
2 \cdot \pair{\rho; x, x-1, \dots, -(x-1)} \rtimes \pi(\phi_0, \eta|\AA_{\phi_0}).
\]
By Proposition \ref{sci}, we have an exact sequence
\[
\begin{CD}
0 @>>> \pi(\phi', \eta'_+) \oplus \pi(\phi', \eta'_-) @>>> \Pi @>>> \sigma @>>> 0,
\end{CD}
\]
where $\Pi = \pair{\rho; x, x-1, \dots, -(x-1)} \rtimes \pi(\phi_0, \eta|\AA_{\phi_0})$,
and $\sigma$ is the unique irreducible quotient of $\Pi$.
Fix $\epsilon \in \{\pm1\}$.
By Lemma \ref{2times}, we see that 
$\Jac_{\rho|\cdot|^x}(\pi(\phi, \eta_\epsilon)) \supset 2 \cdot \pi(\phi', \eta'_\epsilon)$. 
On the other hand, 
since 
$\semi\Jac_{P_{d(2x+1)}}(\pi(\phi, \eta_\epsilon)) \supset \St(\rho, 2x+1) \otimes \pi(\phi_0, \eta|\AA_{\phi_0})$, 
we have 
\[
\semi\Jac_{P_{2dx}}(\Jac_{\rho|\cdot|^x}(\pi(\phi, \eta_\epsilon)))
\supset 
|\cdot|^{-\half{1}} \St(\rho, 2x) \otimes \pi(\phi_0, \eta|\AA_{\phi_0}).
\]
This implies that $\Jac_{\rho|\cdot|^x}(\pi(\phi, \eta_\epsilon))$ 
contains an irreducible non-tempered representation, which must be $\sigma$.
Hence
\[
\Jac_{\rho|\cdot|^x}(\pi(\phi, \eta_\epsilon)) \supset 2 \cdot \pi(\phi', \eta_\epsilon) + \sigma
= \Pi + \pi(\phi', \eta_\epsilon) - \pi(\phi', \eta_{-\epsilon}).
\]
Considering $\Jac_{\rho|\cdot|^x}(\pi(\phi, \eta_+) \oplus \pi(\phi, \eta_-))$, 
we see that this inclusion must be an equality.
\end{proof}

If $x=0$ and $m=2$, 
we see that $\Jac_{\rho}(\pi(\phi, \eta)) \supset \pi(\phi_0, \eta|\AA_{\phi_0})$.
By the same argument, 
this inclusion must be an equality.
This completes the proof of Theorem \ref{jac1} (4),
so that the ones of all statements of Theorem \ref{jac1}.
\par

\subsection{Description of $\mu^*_\rho$}
We prove Theorem \ref{jac2} in this subsection.
To do this, we need the following specious lemma.

\begin{lem}\label{nec}
Let $\phi \in \Phi_\gp(G)$ and $\ub{x} = (x_1, \dots, x_m) \in \R^m$.
Suppose that $\Jac_{\rho|\cdot|^{\ub{x}}}(\pi) \not= 0$ for some $\pi \in \Pi_\phi$.

\begin{enumerate}
\item
If $x_m < 0$, then $x_i = -x_m$ for some $i$. 

\item
Suppose that $\ub{x}$ is of the form
\[
\ub{x} = 
( \underbrace{x_1^{(1)}, \dots, x_{m_1}^{(1)}}_{m_1}, 
\dots, 
\underbrace{x_1^{(k)}, \dots, x_{m_k}^{(k)}}_{m_k} )
\]
with $x_{i-1}^{(j)} > x_{i}^{(j)}$ for $1 \leq j \leq k$, $1 < i \leq m_j$, 
and $x_1^{(1)} \leq \dots \leq x_1^{(k)}$.
Then $x_1^{(j)} \geq 0$ for $j  =1, \dots, k$, and 
\[
\phi \supset \bigoplus_{j=1}^{k} \rho \boxtimes S_{2x_1^{(j)}+1}.
\]
\end{enumerate}
\end{lem}
\begin{proof}
We prove the lemma by induction on $m$.
By Lemma \ref{nonvanish}, we see that
$2x_1+1$ is a positive integer, 
and $\phi$ contains $\rho \boxtimes S_{2x_1+1}$.
In particular, we obtain the lemma for $m=1$.
\par

Suppose that $m \geq 2$ and put $\ub{x'} = (x_2, \dots, x_m) \in \R^{m-1}$.
By Theorem \ref{jac1}, 
one of the following holds.
\begin{itemize}
\item
$\Jac_{\rho|\cdot|^{\ub{x'}}}(\pi') \not= 0$
for some $\pi' \in \Pi_{\phi'}$ with $\phi' = \phi - (\rho \boxtimes S_{2x_1+1}) \oplus (\rho \boxtimes S_{2x_1-1})$; 

\item
$\Jac_{\rho|\cdot|^{\ub{x'}}}(\pair{\rho; x_1, x_1-1, \dots, -(x_1-1)} \rtimes \pi_0) \not= 0$
for some $\pi_0 \in \Pi_{\phi_0}$ with $\phi_0 = \phi - (\rho \boxtimes S_{2x_1+1})^{\oplus2}$.
\end{itemize}
The former case can occur only if $x_1 > 0$, and 
the latter case can occur only if $\phi \supset (\rho \boxtimes S_{2x_1+1})^{\oplus 2}$. 
\par

We consider the former case.
Assume that $x_1 > 0$ and $\Jac_{\rho|\cdot|^{\ub{x'}}}(\pi') \not= 0$
for some $\pi' \in \Pi_{\phi'}$ with $\phi' = \phi - (\rho \boxtimes S_{2x_1+1}) \oplus (\rho \boxtimes S_{2x_1-1})$.
By the induction hypothesis, we have $x_i = -x_m$ for some $i \geq 2$ when $x_m < 0$, and 
\[
\phi' \supset 
(\rho \boxtimes S_{2x_1^{(1)}-1}) \oplus \left(\bigoplus_{j=2}^{k} \rho \boxtimes S_{2x_1^{(j)}+1} \right)
\]
when $\ub{x}$ is of the form in (2) since $\ub{x'}$ is also of the form.
This implies the assertion for $\phi$.
\par

We consider the latter case.
Assume that $\phi \supset (\rho \boxtimes S_{2x_1+1})^{\oplus 2}$, 
and that 
\[
\Jac_{\rho|\cdot|^{\ub{x'}}}(\pair{\rho; x_1, x_1-1, \dots, -(x_1-1)} \rtimes \pi_0) \not= 0
\]
for some $\pi_0 \in \Pi_{\phi_0}$ with $\phi_0 = \phi - (\rho \boxtimes S_{2x_1+1})^{\oplus2}$.
By Corollary \ref{cor.tdc}, 
we can divide 
\[
\{2, \dots, m\} = \{i_1, \dots, i_{m_1}\} \sqcup \{j_{1}, \dots, j_{m_2}\} \sqcup \{k_1, \dots, k_{m_3}\}
\]
with $i_1 < \dots < i_{m_1}$, $j_1< \dots < j_{m_2}$, $k_1 < \dots < k_{m_3}$ and $m_2+m_3 \leq 2x_1$
such that
\begin{itemize}
\item
$\Jac_{\rho|\cdot|^{y_1}, \dots, \rho|\cdot|^{y_{m_1}}}(\pi_0) \not= 0$ with $y_t = x_{i_t}$; 
\item
$x_{j_t} = x_1+1-t$ for $t = 1, \dots, m_2$; 
\item
$x_{k_t} = x_1-t$ for $t = 1, \dots, m_3$.
\end{itemize}
Considering the following four cases, 
we can prove the existence $x_i$ satisfying $x_i = -x_m$ when $x_m < 0$.
\begin{itemize}
\item
When $m = i_{m_1}$, by the induction hypothesis, we have $x_{i_t} = -x_m$ for some $t$.
\item
When $m = j_{m_2}$, 
we have $x_m = x_1+1-m_2 < 0$ so that $x_{j_t} = -x_m$ with $t = 2x_1+2-m_2$.
\item
When $m = k_{m_3}$ and $m_3 < 2x_1$, 
we have $x_m = x_1-m_3 < 0$ so that $x_{k_t} = -x_m$ with $t = 2x_1-m_3$.
\item
When $m = k_{m_3}$ and $m_3 = 2x_1$, we have $x_1 = -x_m$.
\end{itemize}
On the other hand, 
when $\ub{x}$ is of the form in (2), 
since $x_1^{(j)} \geq x_1^{(1)} = x_1$, 
there is at most one $j_0 \geq 2$ such that 
\[
x_1^{(j_0)} \in \{x_{j_t}\ |\ t = 1, \dots, m_2\} \cup \{x_{k_t}\ |\ t = 1, \dots, m_3\}, 
\]
in which case, $x_1^{(j_0)} = x_1$.
By the induction hypothesis, we have
\[
\phi_0 \supset \bigoplus_{\substack{2 \leq j \leq k \\ j \not= j_0}} \rho \boxtimes S_{2x_1^{(j)}+1}.
\]
This implies the assertion for $\phi$.
This completes the proof.
\end{proof}


Now we can prove Theorem \ref{jac2}.
\begin{proof}[Proof of Theorem \ref{jac2}]
Since the subgroup of $\RR_m$ spanned by $\Irr_\rho(\GL_{dm}(F)) =  \{\tau_{\ub{x}}\ |\ \ub{x} \in \Omega_m\}$ 
has another basis $\{\Delta_{\ub{x}}\ |\ \ub{x} \in \Omega_m\}$, 
we can write 
\[
\mu^*_\rho(\pi) = \sum_{m \geq 0}\sum_{\ub{x} \in \Omega_m} 
\Delta_{\ub{x}} \otimes \Pi_{\ub{x}}(\pi)
\]
for some virtual representation $\Pi_{\ub{x}}(\pi)$.
For $\ub{y} \in \Omega_m$, 
applying $\Jac_{\rho|\cdot|^{\ub{y}}}$ to $\semi\Jac_{P_{dm}}(\pi)$, 
we have
\[
\Jac_{\rho|\cdot|^{\ub{y}}}(\pi) 
= \sum_{\ub{x} \in \Omega_m} 
\Jac_{\rho|\cdot|^{\ub{y}}}(\Delta_{\ub{x}}) \otimes \Pi_{\ub{x}}(\pi)
= \sum_{\ub{x} \in \Omega_m} m(\ub{y}, \ub{x}) \cdot \Pi_{\ub{x}}(\pi), 
\]
where $m(\ub{y}, \ub{x}) = \dim_\C \Jac_{\rho|\cdot|^{\ub{y}}}(\Delta_{\ub{x}})$.
If $m'(\ub{x'}, \ub{y}) \in \mathbb{Q}$ satisfies that 
\[
\sum_{\ub{y} \in \Omega_m} m'(\ub{x'}, \ub{y}) m(\ub{y}, \ub{x}) = \delta_{\ub{x'}, \ub{x}}, 
\]
we have
\[
\Pi_{\ub{x}}(\pi) = \sum_{\ub{y} \in \Omega_m} m'(\ub{x}, \ub{y}) \cdot \Jac_{\rho|\cdot|^{\ub{y}}}(\pi). 
\]
Hence we have
\[
\mu^*_\rho(\pi) = \sum_{m \geq 0}\sum_{\ub{x}, \ub{y} \in \Omega_m} 
m'(\ub{x}, \ub{y}) \cdot \Delta_{\ub{x}} \otimes \Jac_{\rho|\cdot|^{\ub{y}}}(\pi).
\]
\par

One can easily to prove by induction that $m'(\ub{x}, \ub{x}(\ub{k})) = 0$ 
unless $\ub{x} = \ub{x}(\ub{k'})$ for some $\ub{k'} \in K_\phi$
(see also the proof of Lemma \ref{dim}).
Therefore, by Lemma \ref{nec}, we have
\[
\mu^*_\rho(\pi) = \sum_{m \geq 0}\sum_{\ub{k}, \ub{k'} \in K_\phi} 
m'(\ub{x}(\ub{k'}), \ub{x}(\ub{k})) \cdot \Delta_{\ub{x}(\ub{k'})} \otimes \Jac_{\rho|\cdot|^{\ub{x}(\ub{k})}}(\pi).
\]
This completes the proof of Theorem \ref{jac2}.
\end{proof}

\subsection{A remark on standard modules}
As a consequence of Theorem \ref{jac1}, 
we can prove the irreducibility of certain standard modules.
\begin{cor}\label{std}
Let $\phi \in \Phi_\gp(G)$ and $\eta \in \widehat{\AA_\phi}$ such that $\pi(\phi, \eta) \not= 0$.
Suppose that $\phi \supset \rho \boxtimes S_{2x+1}$ for $x>0$ but $\Jac_{\rho|\cdot|^x}(\pi(\phi, \eta)) = 0$.
Then the standard module 
\[
\Pi = \pair{\rho; x, x-1, \dots, -(x-1)} \rtimes \pi(\phi, \eta)
\] 
is irreducible.
\end{cor}
\begin{proof}
Let $\sigma$ be the unique irreducible quotient of $\Pi$, which is non-tempered.
By the same argument as the proof of Proposition \ref{sci}, 
we see that $\sigma$ is the unique irreducible non-tempered subquotient of $\Pi$.
Suppose that $\Pi$ is reducible.
If $\pi'$ is another irreducible subquotient of $\Pi$, 
by considering its cuspidal support or its Plancherel measure, 
we see that $\pi' \in \Pi_{\phi'}$ 
with $\phi' = \phi \oplus (\rho \boxtimes S_{2x-1}) \oplus (\rho \boxtimes S_{2x+1})$.
Since $\phi \supset \rho \boxtimes S_{2x+1}$, 
we see that $\phi' \supset (\rho \boxtimes S_{2x+1})^{\oplus 2}$.
By Theorem \ref{jac1}, $\Jac_{\rho|\cdot|^x}(\pi')$ contains an irreducible non-tempered representation.
However, $\Jac_{\rho|\cdot|^x}(\Pi) = \St(\rho, 2x-1) \rtimes \pi(\phi, \eta)$ consists of tempered representations.
This is a contradiction.
\end{proof}

\begin{ex}\label{246}
Consider $\phi = S_2 \oplus S_4 \oplus S_6 \in \Phi_\disc(\SO_{13})$
and $\eta \in \widehat{A_\phi}$ given by $\eta(\alpha_{S_{2a}}) = (-1)^{a}$ for $a=1,2,3$.
Then $\pi(\phi, \eta)$ is an irreducible supercuspidal representation.
Moreover, the standard module 
\[
\Pi = \pair{\1_{\GL_1(F)}; \half{5}, \half{3}, \half{1}, -\half{1}, -\half{3}} \rtimes \pi(\phi, \eta)
\]
of $\SO_{23}(F)$ is irreducible.
Note that the $L$-parameter 
$\phi' = \phi \oplus |\cdot|^{\half{1}} S_5 \oplus |\cdot|^{-\half{1}} S_5$
of $\Pi$ is non-generic
since 
\[
L(s, \phi', \Ad) = \zeta_F(s-1)\zeta_F(s)^{3}\zeta_F(s+1)^{13}
\zeta_F(s+2)^{10}\zeta_F(s+3)^{12}\zeta_F(s+4)^{5}\zeta_F(s+5)^{3} 
\]
has a pole at $s=1$, 
where $\zeta_F$ is the local zeta function associated to $F$.
In particular, the standard module
\[
\Pi_0 = \pair{\1_{\GL_1(F)}; \half{5}, \half{3}, \half{1}, -\half{1}, -\half{3}} \rtimes \pi(\phi, \1)
\]
is reducible by Theorem \ref{smc}.
\end{ex}


\end{document}